\documentclass{amsart}
\usepackage[latin1]{inputenc}
\usepackage{amsmath}
\usepackage{amsthm}
\usepackage{amsfonts}
\usepackage{amssymb}
\usepackage[T1]{fontenc}
\usepackage{float}
\usepackage{color}
\usepackage{lscape}
\usepackage[all]{xy}
\usepackage[notref,notcite]{showkeys}
\usepackage{hyperref}
\usepackage{calrsfs}

\theoremstyle{plain}
\newtheorem{theorem}{Theorem}[section]
\newtheorem{prop}[theorem]{Proposition}
\newtheorem{lem}[theorem]{Lemma}
\newtheorem{corol}[theorem]{Corollary}
\newtheorem{conj}[theorem]{Conjecture}

\theoremstyle{definition}
\newtheorem{defi}[theorem]{Definition}

\newtheorem{rmq}[theorem]{Remark}
\newtheorem{exmp}[theorem]{Example}

\def\coker{{rm{coker}\,}}
\def\Gr{{\rm{Gr}}}

\def\St{{\rm{St}}}

\def\dim{{\rm{dim}\,}}
\def\dimk{{\rm{dim}_\k}}
\def\ddim{{\mathbf{dim}\,}}
\def\ddimC{{\mathbf{dim}_{\CC}\,}}
\def\rep{{\rm{rep}}}

\def\<{\left<}
\def\>{\right>}

\def\d{{\mathbf d}}
\def\k{{\mathbf{k}}}
\def\kQ{\mathbf{k}Q}

\def\ens#1{\left\{ #1 \right\}}

\def\fl{{\longrightarrow}\,}

\def\CC{{\mathcal{C}}}

\def\Q{{\mathbb{Q}}}
\def\Z{{\mathbb{Z}}}

\def\T{{\mathcal {T}}}

\def\add{{\rm{add}}\,}
\def\ind{{\rm{ind}}}
\def\coind{{\rm{coind}}}

\def\split{{\rm{split}}}

\def\ens#1{\left\{ #1 \right\}}
\def\Ext{{\rm{Ext}}}
\def\End{{\rm{End}}}
\def\Hom{{\rm{Hom}}}
\def\Aut{{\rm{Aut}}}
\def\Ob{{\rm{Ob}}}

\def\cone{{\rm{Cone}}\,}

\def\modg{{\textrm{-mod}\,}}

\def\projg{{\textrm{-proj}\,}}

\def\id{{\text{id}}}

\def\op{{\rm{op}\,}}
\def\min{{\rm{min}\,}}

\def\Alin#1{\overrightarrow{\mathbb A}_{#1}}
\def\repBA#1{\rep(B_T\overrightarrow{\mathbb A}_{4},#1)}
\def\ker{{\rm{Ker}}}
\def\coker{{\rm{Coker}}}

\title{Generic cluster characters}
\author{G. Dupont}
\address{
	Universit\'e de Sherbrooke\\
	2500, Boul. de l'Universit\'e\\
	J1K 2R1, Sherbrooke, QC, Canada.
}
\email{gregoire.dupont@usherbrooke.ca}
\urladdr{
	http://pages.usherbrooke.ca/gdupont2
}

\begin{document}

\begin{abstract}
	Let $\CC$ be a Hom-finite triangulated 2-Calabi-Yau category with a cluster-tilting object $T$. Under a constructibility condition we prove the existence of a set $\mathcal G^T(\CC)$ of generic values of the cluster character associated to $T$. If $\CC$ has a cluster structure in the sense of Buan-Iyama-Reiten-Scott, $\mathcal G^T(\CC)$ contains the set of cluster monomials of the corresponding cluster algebra. Moreover, these sets coincide if $\mathcal C$ has finitely many indecomposable objects. 

	When $\CC$ is the cluster category of an acyclic quiver and $T$ is the canonical cluster-tilting object, this set coincides with the set of generic variables previously introduced by the author in the context of acyclic cluster algebras. In particular, it allows to construct $\Z$-linear bases in acyclic cluster algebras.
\end{abstract}

\maketitle

\setcounter{tocdepth}{1}
\tableofcontents

\section{Introduction and main results}
	\subsection{Cluster algebras}
		Cluster algebras were introduced in \cite{cluster1} in order to design a combinatorial framework for studying total positivity in algebraic groups and canonical bases in quantum groups. It is expected, and proved in some cases, that a cluster algebra has distinguished linear bases providing combinatorial models for canonical or semicanonical bases in quantum groups \cite{cluster1,GLS:KMgroups,Nakajima:cluster, Lampe:Kronecker}. Besides this question, cluster algebras have shown interactions with various areas of mathematics like Lie theory, combinatorics, Teichm\"uller theory, Poisson geometry or representation theory of algebras. This last connection with representation theory was particularly fruitful for the construction of linear bases in a large class of cluster algebras \cite{CZ,CK1,Cerulli:A21,Dupont:BaseAaffine,DXX:basesv3,GLS:KMgroups,GLS:generic}. 

		In full generality, a (coefficient-free) \emph{cluster algebra} can be associated to any pair $(Q,\mathbf x)$ where $Q=(Q_0,Q_1)$ is a quiver and $\mathbf x=(x_i|i \in Q_0)$ is a $Q_0$-tuple of indeterminates over $\Z$. By a quiver $Q=(Q_0,Q_1)$, we always mean an oriented graph such that $Q_0$ is the (finite) set of vertices and $Q_1$ is the (finite) set of arrows. Moreover, we always assume that a quiver $Q$ does not contain any loops or 2-cycles. Given such a pair $(Q,\mathbf x)$, we denote by $\mathcal A(Q,\mathbf x)$ the corresponding cluster algebra. 

		$\mathcal A(Q,\mathbf x)$ is a subalgebra of the algebra $\Z[\mathbf x^{\pm 1}]$ of Laurent polynomials in the $x_i$ with $i \in Q_0$. It is equipped with a distinguished set of generators called \emph{cluster variables}, gathered into possibly overlapping sets of fixed cardinality called \emph{clusters}, generated by a recursive process called \emph{mutation}. Monomials in variables belonging all to a same cluster are called \emph{cluster monomials} and we denote by $\mathcal M(Q,\mathbf x)$ the set of all cluster monomials in $\mathcal A(Q,\mathbf x)$. 

		A \emph{$\Z$-basis} in the cluster algebra $\mathcal A(Q,\mathbf x)$ is a free generating set of $\mathcal A(Q,\mathbf x)$ viewed as a $\Z$-module. If $\mathcal A(Q,\mathbf x)$ has finitely many cluster variables, then the set $\mathcal M(Q,\mathbf x)$ of cluster monomials is a $\Z$-basis in $\mathcal A(Q,\mathbf x)$ \cite{CK1}. If $\mathcal A(Q,\mathbf x)$ has infinitely many cluster variables it was observed that cluster monomials do not span the cluster algebra as a $\Z$-module \cite{shermanz}. However, it is conjectured in full generality, and proved in several cases (see for instance \cite{GLS:KMgroups,Plamondon:ClusterAlgebras}), that cluster monomials are linearly independent over $\Z$. The aim of this article is to provide, for a wide class of cluster algebras, a general construction of a distinguished family of elements in $\Z[\mathbf x^{\pm 1}]$ containing naturally the set of cluster monomials in $\mathcal A(Q,\mathbf x)$ and which is expected to form a $\Z$-basis in the cluster algebra $\mathcal A(Q,\mathbf x)$.

	\subsection{Triangulated 2-Calabi-Yau realisations}
		Let $\CC$ be a triangulated 2-Calabi-Yau category over an algebraically closed field $\k$ such that cluster-tilting subcategories in $\CC$ determine a cluster structure in the sense of \cite{BIRS} (see Section \ref{ssection:clusterstructures} for details). Let $T$ be a cluster-tilting object in $\CC$. It is known that the set of indecomposable rigid objects which are reachable from $T$ is in bijection with the set of cluster variables in $\mathcal A(Q_{T},\mathbf x)$ where $Q_{T}$ is the ordinary quiver of the so-called \emph{2-Calabi-Yau tilted} algebra $\End_{\CC}(T)^{\op}$ \cite{BIRS} (see Section \ref{section:preliminaires} for details). This bijection can be made explicit using the so-called \emph{cluster character} 
		$$X^{T}_?:\Ob(\CC) \fl \Z[\mathbf x^{\pm 1}]$$
		introduced in \cite{Palu} whose definition is recalled in Section \ref{ssection:characters}. 

		When $\CC=\CC_Q$ is the cluster category of an acyclic quiver and $T=\kQ$ is the canonical cluster-tilting object, the cluster character $X^{T}_?$ coincides with the \emph{Caldero-Chapoton map} $CC$ introduced in \cite{CC,CK2}. In \cite{Dupont:genericvariables}, the author introduced and studied generic values of restrictions of the Caldero-Chapoton map to representation spaces under the name of \emph{generic variables}. It is known that these generic variables form a $\Z$-basis in the cluster algebra $\mathcal A(Q,\mathbf x)$ \cite{Dupont:BaseAaffine, DXX:basesv3,GLS:generic}. In this article, we generalise this construction to cluster algebras which can be realised with triangulated 2-Calabi-Yau categories and we conjecture that it still provides a $\Z$-linear basis of the corresponding cluster algebra.
	
	\subsection{Main results}
		If $\k$ is the field of complex numbers and $\CC$ has constructible cones with respect to $\add T$-morphisms (see Section \ref{ssection:constructibility} for details), we associate to any element $\gamma \in K_0(\add T)$ a Laurent polynomial $X(\gamma)$ by taking the character of the cone of a generic morphism in $\End_{\CC}(T)^{\op}$-mod with $\delta$-vector $\gamma$ in the sense of \cite{DF:generalpresentations}. The set 
		$$\mathcal G^{T}(\CC)=\ens{X(\gamma)|\gamma \in K_0(\add T)} \subset \Z[\mathbf x^{\pm 1}]$$
		is called the set of \emph{generic characters associated to $T$ in $\CC$}. 

		Theorem \ref{theorem:clustermonomials} asserts that the set $\mathcal G^{T}(\CC)$ of generic characters contains naturally the set of cluster monomials (and thus of cluster variables) in the cluster algebra $\mathcal A(Q_{T},\mathbf x)$.  

		If $\CC$ has finitely many indecomposable objects, we prove that the set $\mathcal G^{T}(\CC)$ coincides with the set $\mathcal M(Q_{T},\mathbf x)$ and thus that provides a $\Z$-linear basis of the finite type cluster algebra $\mathcal A(Q_{T},\mathbf x)$ (Theorem \ref{theorem:finitetype} and Corollary \ref{corol:basetypefini}).

		When $\CC=\CC_Q$ is the cluster category of an acyclic quiver $Q$ and $T=\kQ$ is the canonical cluster-tilting object, we prove that $\mathcal G^{T}(\CC)$ satisfies multiplicative properties compatible with the virtual generic decomposition of \cite{IOTW} (Theorem \ref{theorem:multvirtual}). 

		In this case, we also prove that the set $\mathcal G^{T}(\CC)$ coincides with the set $\mathcal G(Q)$ of generic variables introduced in \cite{Dupont:BaseAaffine}. In particular, it provides a $\Z$-basis in the acyclic cluster algebra $\mathcal A(Q,\mathbf x)$ (Proposition \ref{prop:XetCC} and Corollary \ref{corol:acyclic}).

	\subsection{Organisation of the paper}
		In Section \ref{section:preliminaires}, we recall the necessary background concerning triangulated 2-Calabi-Yau realisations of cluster algebras, cluster structures, cluster characters and constructibility of cones. In Section \ref{section:genericcharacters}, we define generic cluster characters in full generality and connect this construction to general presentations of modules introduced in \cite{DF:generalpresentations}. In Section \ref{section:monomials} we prove that the set of generic cluster characters naturally contains the cluster monomials of the corresponding cluster algebra. In the sequel, we focus on the case of cluster categories associated to acyclic quivers. In Section \ref{section:index}, we relate indices and dimension vectors in cluster categories in order to describe a natural parametrisation of generic characters using $\mathbf g$-vectors. In Section \ref{section:genericdcp}, we prove the compatibility of generic characters with the virtual generic decomposition of \cite{IOTW}. In Section \ref{section:XetCC}, we prove that this construction indeed generalises the construction of generic variables given in \cite{Dupont:genericvariables} and provides a $\Z$-basis for acyclic cluster algebras. Section \ref{section:example} presents a detailed example in a cluster category of type $A_3$ and finally, Section \ref{section:conjectures} states conjectures and questions relative to this construction.
	
\section{Preliminaries}\label{section:preliminaires}
	Throughout the article, $\k$ is the field of complex numbers.
	\subsection{Triangulated 2-Calabi-Yau categories}\label{ssection:2CY}
		 Without other specification, $\CC$ will always denote a $\k$-linear triangulated category with suspension functor $[1]$ which is assumed to be~:
		\begin{itemize}
			\item \emph{Hom-finite}~: $\dimk \Hom_{\CC}(M,N)<\infty$ for any objects $M,N$ in $\CC$~;
			\item \emph{2-Calabi-Yau}~: there is a bifunctorial isomorphism
				$$\Hom_{\CC}(M,N) \simeq D\Hom_{\CC}(N,M[2])$$
				for any two objects $M,N$ in $\CC$ where $D = \Hom_{\k}(-,\k)$ is the standard duality. 
			\item with split idempotents.
		\end{itemize}

		A \emph{cluster-tilting subcategory} in $\CC$ is a full additive subcategory $\mathcal T$ of $\mathcal C$ which is stable under direct factors and such that~:
		\begin{itemize}
			\item the functors $\Hom_{\CC}(X,-):\mathcal T \fl \k\modg$ and $\Hom_{\CC}(-,X):\mathcal T^{\op} \fl \k\modg$ are finitely presented for any object $X$ in $\CC$~;
			\item an object $X$ in $\mathcal C$ belongs to $\mathcal T$ if and only if $\Ext^1_{\mathcal C}(T,X)=0$ for any object $T$ in $\mathcal T$.
		\end{itemize}

		A \emph{cluster-tilting object} $T$ in $\CC$ is a basic object such that $\add T$ is a cluster-tilting subcategory of $\mathcal C$. Equivalently, $T$ is a cluster-tilting object if it is basic, rigid (that is, without self-extension) and such that for any $X$ in $\CC$, $\Ext^1_{\CC}(T,X)=0$ implies $X \in \add T$. 

		Throughout the article, we always assume that $\CC$ contains a cluster-tilting object $T$ with $n$ distinct indecomposable summands. We denote by $Q_{T}$ the ordinary quiver of the 2-Calabi-Yau-tilted algebra $B_{T}=\End_{\CC}(T)^{\op}$ and by $\mathcal T$ the additive category $\add T$. Any two cluster-tilting objects giving rise to the same cluster-tilting subcategory are isomorphic and thus, up to isomorphism, the quiver $Q_T$ only depends on $\mathcal T$ and we will sometimes denote it by $Q_{\mathcal T}$. 

		\begin{exmp}
			Let $Q$ be an acyclic quiver and let $\CC_Q$ be the \emph{cluster category of $Q$} introduced in \cite{BMRRT} (see also \cite{CCS1} for quivers of Dynkin type $\mathbb A$). Then $\CC_Q$ satisfies all the above hypotheses. Moreover, the path algebra $\kQ$ of $Q$ can be identified with a cluster-tilting object in $\CC_Q$ and it is called the \emph{canonical cluster-tilting object in $\CC_Q$}.
		\end{exmp}

		\begin{exmp}
			Let $(Q,W)$ be a quiver with potential in the sense of \cite{DWZ:potentials}. Assume that $(Q,W)$ is Jacobi-finite, that is, the corresponding Jacobian algebra $\mathcal J_{(Q,W)}$ is finite dimensional. Then the \emph{generalised cluster category} $\CC_{(Q,W)}$ introduced in \cite[\S 3]{Amiot:clustercat} satisfies all the above hypotheses. Moreover, it contains a cluster-tilting object $T$ such that the corresponding 2-Calabi-Yau-tilted algebra is isomorphic to $\mathcal J_{(Q,W)}$.
		\end{exmp}
		
	\subsection{$\mathcal T$-morphisms}\label{ssection:Tmorphisms}
		\begin{defi}
			A \emph{$\mathcal T$-morphism} in $\CC$ is a morphism $f \in \Hom_{\CC}(T_1,T_0)$ for some objects $T_0,T_1$ in $\mathcal T$.
		\end{defi}

		We denote by $F_{T}$ the functor 
		$$F_{T}=\Hom_{\CC}(T,-): \CC \fl B_{T}\modg$$
		inducing an equivalence of categories 
		$$\CC/(T[1]) \xrightarrow{\sim} B_{T}\modg$$
		where $B_{T}\modg$ is the category of finitely generated left $B_{T}$-modules \cite{BMR1,KR:clustertilted}. Note that projective $B_{T}$-modules are given by the $F_{T}M$ where $M$ runs over $\mathcal T$.

		We denote by $K_0(\mathcal T)$ the Grothendieck group of the additive category $\mathcal T$ and for any object $M$ in $\mathcal T$, we denote by $[M]$ its class in $K_0(\mathcal T)$.

		As shown in \cite{KR:acyclic}, for any object $M$ in $\CC$, there exist (non-unique) triangles
		$$T_1^M \fl T_0^M \fl M \fl T_1^M[1],$$
		$$M \fl T_M^0[2] \fl T_M^1[2] \fl M[1]$$
		with $T_i^M,T^i_M$ in $\mathcal T$ for any $i \in \ens{1,2}$.

		Following \cite{Palu}, the \emph{index} of $M$ (with respect to $\mathcal T$) is
		$$\ind_{\mathcal T}(M)=[T_0^M]-[T_1^M]$$
		and the \emph{coindex} of $M$ (with respect to $\mathcal T$) is
		$$\coind_{\mathcal T}(M)=[T^0_M]-[T^1_M].$$
		Note that the index and the coindex of $M$ are well-defined elements in $K_0(\mathcal T)$ in the sense that they do not depend on the choice of the above triangles \cite[Lemma 2.1]{Palu}. In particular, the map $T_0^M \fl M$ (resp. $M \fl T^M_0[2]$) is not necessarily a minimal right $\mathcal T$-approximations (resp. left $\mathcal T[2]$-approximation).

	\subsection{Cluster structures}\label{ssection:clusterstructures}
		The notion of cluster structure was first introduced in \cite{BIRS} in order to design a framework for cluster mutations in 2 Calabi-Yau triangulated categories. We thus say that \emph{cluster-tilting subcategories of $\mathcal C$ determine a cluster structure on $\mathcal C$} if~:
		\begin{enumerate}
			\item For each cluster-tilting subcategory $\mathcal T'$ of $\CC$ and each indecomposable object $M$ of $\mathcal T'$, there is a unique (up to isomorphism) indecomposable $M^*$ not isomorphic to $M$ such that the additive subcategory $\mu_M(\mathcal T')$ of $\mathcal C$ with set of indecomposables $\ind(\mathcal T' \setminus \ens {M}) \sqcup \ens{M^*}$ is a cluster-tilting subcategory, called \emph{mutation of $\mathcal T'$ at $M$}.
			\item In the above situation, there are triangles
			$$M^* \xrightarrow f E \xrightarrow g M \rightarrow M^*[1] \textrm{ and } M \xrightarrow{f'} E' \xrightarrow{g'} M^* \rightarrow M[1]$$
			where $f,f'$ are minimal left $\mathcal T' \cap \mu_M(\mathcal T')$-approximations and $g,g'$ are minimal right $\mathcal T' \cap \mu_M(\mathcal T')$-approximations.
			\item For any cluster-tilting subcategory $\mathcal T'$ of $\CC$, the quiver $Q_{\mathcal T'}$ has no loops or 2-cycles.
			\item For any cluster-tilting subcategory $\mathcal T'$ of $\CC$ and any indecomposable object $M$ in $\mathcal T'$, we have $Q_{\mu_M(\mathcal T')}=\mu_M(Q_{\mathcal T})$.
		\end{enumerate}

		The \emph{mutation} of a cluster-tilting object $T'$ in such a category $\CC$ is defined via the mutation of the corresponding cluster-tilting subcategory $\add T'$. We say that a cluster-tilting subcategory $\mathcal T'$ in $\CC$ is \emph{reachable from $T$} if it can be obtained from $\add T$ by a finite number of mutations. An arbitrary object $M$ in $\CC$ is called \emph{reachable from $T$} if it belongs to  a cluster-tilting subcategory of $\CC$ which is reachable from $\mathcal T$. In particular, an object reachable from $T$ is always rigid. 

		\begin{exmp}
			If $Q$ is an acyclic quiver, $\CC=\CC_Q$ is the cluster category of $Q$ and $T=\kQ$, then it is known that cluster-tilting subcategories determine a cluster structure on $\mathcal C$ and that every rigid object is reachable from $T$ \cite{BMRRT,BMR2}.
		\end{exmp}

		\begin{exmp}
			If $(Q,W)$ is a Jacobi-finite quiver with potential which is non-degenerate, then the cluster-tilting subcategories of $\CC_{(Q,W)}$ form a cluster structure on $\CC_{(Q,W)}$ \cite{Amiot:clustercat}. Moreover, if $(Q,W)$ comes from an unpunctured surface in the sense of \cite{ABCP}, then every rigid object is reachable from any cluster-tilting object in $\CC_{(Q,W)}$ \cite{BZ:clustercatsurfaces}.
		\end{exmp}

	\subsection{Constructible cones}\label{ssection:constructibility}
		Given a morphism $f \in \Hom_{\CC}(M,N)$ with $M,N$ objects in $\CC$, the \emph{cone} of $f$ is the unique (up to isomorphism) object $\cone(f)$ in $\CC$ such that there exists a triangle
		$$M \xrightarrow{f} N \fl \cone(f) \fl M[1].$$

		For the geometric considerations in this article we will moreover always assume that the category $\CC$ has \emph{constructible cones with respect to $\mathcal T$-morphisms}. This notion is a weak analogue to the constructibility of the cones first introduced in \cite{Palu:multiplication}. We recall it here.

		Let $\Alin 4$ be the quiver $1\fl 2\fl 3\fl 4$. For any $\d=(d_1,d_2,d_3,d_4) \in \Z_{\geq 0}^4$, a \emph{$\d$-dimensional matrix representation} of $\Alin 4$ in $B_T$-mod is given by~:
		\begin{enumerate}
			\item a left $B_T$-module structure on $\k^{d_i}$ for every $1 \leq i \leq 4$~;
			\item a $B_T$-linear map $\k^{d_i} \fl \k^{d_{i+1}}$ for any $1 \leq i \leq 3$.
		\end{enumerate}
		The set of $\d$-dimensional representations in $B_T$-mod of $\Alin 4$ form an affine variety, denoted by $\repBA{\d}$. The group $GL(\d) = \prod_{i=1}^4 GL(d_i,\k)$ acts by base change and the set of orbits is denoted by $\repBA{\d}/GL(\d)$.

		Let $T_0$ and $T_1$ be objects in $\mathcal T$. Let 
		$$\d_0 = (\dimk F_T T_1, \dimk F_TT_0, \dimk F_T T_0 -\dimk F_T T_1, \dimk F_TT_1[1])$$
 		and let $\Phi_{T_1,T_0}$ be the map from $\Hom_{\CC}(T_1,T_0)$ to
 		$$
 		\coprod_{\d\leq \d_0} \repBA{\d}/GL(\d)
 		$$
 		sending a morphism $T_1 \xrightarrow f T_0$ to the orbit of the exact sequence of $B_T$-modules
		$$F_T T_1 \xrightarrow{F_T f} F_T T_0 \xrightarrow{F_T p} F_T \cone(f) \xrightarrow{F_T \epsilon} F_T T_1[1]$$
		where $T_1 \xrightarrow{f} T_0 \xrightarrow{p} \cone(f) \xrightarrow{\epsilon} T_1[1]$ is a triangle in $\CC$. If the map $\Phi_{T_1,T_0}$ lifts to a constructible map 
		$$\phi_{T_1,T_0} : \Hom_{\CC}(T_1,T_0) \fl \coprod_{\d\leq \d_0} \repBA{\d}$$
		for any objects $T_1$ and $T_0$ in $\mathcal T$, then we say that $\CC$ has \emph{constructible cones with respect to $\mathcal T$-morphisms}.

		\begin{exmp}
			Let $(Q,W)$ be a Jacobi-finite quiver with potential, let $\CC_{(Q,W)}$ be the associated generalised cluster category and let $T$ be any cluster-tilting object in $\CC_{(Q,W)}$. Then a direct adaptation of \cite[\S 2.5]{Palu:multiplication} shows that $\CC_{(Q,W)}$ has constructible cones with respect to $\mathcal T$-morphisms. If in particular $Q$ is acyclic and $T$ is any cluster-tilting object in the cluster category $\CC_Q$, then $\CC_Q$ has constructible cones with respect to $\mathcal T$-morphisms.
		\end{exmp}

	\subsection{Cluster characters}\label{ssection:characters}
		From now on, $\mathbf x=(x_1, \ldots, x_n)$ denotes a $n$-tuple of indeterminates over $\Z$ where $n$ is still the number of indecomposable direct summands of the considered cluster-tilting object $T$ in $\CC$. We use the short-hand notation $\Z[\mathbf x^{\pm 1}]$ for the ring $\Z[x_1^{\pm 1}, \ldots, x_n^{\pm 1}]$ of Laurent polynomials in the variables $x_1, \ldots, x_n$.

		Given a $B_{T}$-module $M$ and $\mathbf e \in K_0(B_{T}\modg)$, we denote by $\Gr_{\mathbf e}(M)$ the variety of sub-$B_{T}$-modules of $M$ whose class in $K_0(B_{T}\modg)$ equals $\mathbf e$. This is a projective variety and we denote by $\chi(\Gr_{\mathbf e}(M))$ its Euler-Poincar\'e characteristic with respect to the singular cohomology with rational coefficients.

		Let $\<-,-\>$ be the bilinear form on the split Grothendieck group $K_0(B_{T}\modg)^{\split}$ induced by 
		$$\<M,N\>=\dimk \Hom_{B_{T}}(M,N)-\dimk \Ext^1_{B_{T}}(M,N)$$
		for any two $B_T$-modules $M$ and $N$. It is well-defined on the Grothendieck group $K_0(B_{T}\modg)$ if $B_{T}$ is hereditary but not in general.

		For any $i \in \ens{1, \ldots, n}$, let $S_i$ be the simple $B_{T}$-module associated to $i$. Then the linear form
		$$\<S_i,-\>_a:M \mapsto \<S_i,M\>-\<M,S_i\>$$
		is well-defined on $K_0(B_{T}\modg)$ \cite[Lemma 1.3]{Palu}.

		\begin{defi}[\cite{Palu}]
			The \emph{cluster character associated to $T$} is the map 
			$$X^{T}_?: \Ob(\CC) \fl \Z[\mathbf x^{\pm 1}]$$ defined as follows.
			If $M$ is indecomposable in $\CC$ then 
			$$X^{T}_M=\left\{\begin{array}{l}
					x_i \textrm{ if } M \simeq T_i[1]~; \\
					\displaystyle \sum_{\mathbf e \in K_0(B_{T}\modg)} \chi(\Gr_{\mathbf e}(F_{T}M)) \prod_{i=1}^n x_i^{\<S_i,\mathbf e\>_a-\<S_i,F_{T}M\>} \textrm{ otherwise.}
			\end{array}\right.$$
			and for any two objects $M_1,M_2$ in $\CC$, we set 
			$$X^T_{M_1 \oplus M_2} = X^T_{M_1}X^T_{M_2}.$$
		\end{defi}

		If cluster-tilting subcategories in $\CC$ determine a cluster structure on $\CC$, it is proved in \cite{FK} that the set of $X^{T}_M$ where $M$ runs over the isoclasses of indecomposable rigid objects in $\CC$ which are reachable from $T$ equals the set of cluster variables in $\mathcal A(Q_{T},\mathbf x)$ and moreover,
		$$\ens{X^{T}_M\, |\, M \textrm{ is rigid and reachable from $T$ in $\CC$}}=\mathcal M(Q_{T},\mathbf x).$$

\section{Generic cluster characters for 2-Calabi-Yau triangulated categories}\label{section:genericcharacters}
	\subsection{Generic characters}
		In \cite{Dupont:genericvariables}, we observed that the Caldero-Chapoton map takes generic values on the representation varieties associated to $Q$. For an arbitrary triangulated 2-Calabi-Yau category $\CC$, there is in general no obvious ``nice'' geometry on $\Ob(\CC)$. However, since $\CC$ is $\k$-linear and Hom-finite, Hom-spaces in $\CC$ are finite dimensional $\k$-vector spaces and in particular, they are irreducible affine varieties. Thus, for geometric statements, it is more convenient to consider morphisms instead of objects. This philosophy was for instance already suggested in \cite{IOTW,DF:generalpresentations}.

		\begin{defi}
			For any objects $T_0,T_1$ in $\mathcal T$ and $f \in \Hom_{\CC}(T_1,T_0)$, we set 
			$$X:\left\{\begin{array}{rcl}
			           	\Hom_{\CC}(T_1,T_0) & \fl & \Z[\mathbf x^{\pm 1}]\\
					f & \mapsto & X^{T}_{\cone(f)}
			\end{array}\right.$$
		\end{defi}

		The group $\Aut_{\CC}(T_0)^{\op} \times \Aut_{\CC}(T_1)$ acts on $\Hom_{\CC}(T_1,T_0)$ by
		$$(g_0,g_1).f=g_0 f g_1$$
		and the map $X$ is invariant under this action. 

		If $M$ is a rigid object in $\CC$, consider a triangle 
		$$T^M_1 \xrightarrow{f} T^M_0 \fl M \fl T^M_1[1]$$
		with $T^M_0,T^M_1$ in $\mathcal T$. Then, it follows from \cite[\S 2.1]{DK:2CY} that the orbit of $f$ under this action is a dense open subset in $\Hom_{\CC}(T^M_1,T^M_0)$ so that $X$ is constant over a dense open subset of $\Hom_{\CC}(T^M_1,T^M_0)$. The following lemma proves that such a dense open subset actually exists for any two objects $T_0,T_1$ in $\mathcal T$.

		\begin{lem}\label{lem:Xgen}
			Let $T_0,T_1$ be objects in $\mathcal T$. Then there exists a unique Laurent polynomial $X(T_1,T_0) \in \Z[\mathbf x^{\pm 1}]$ such that $X$ is constant equal to $X(T_1,T_0)$ on a Zariski dense open subset $U_{(T_1,T_0)} \subset \Hom_{\CC}(T_1,T_0)$.
		\end{lem}
		\begin{proof}
			Let $T_0,T_1 \in \mathcal T$ and let $f \in \Hom_{\CC}(T_1,T_0)$. With the assumption on the constructibility of the cones, it follows from \cite[Corollary 6]{Palu:multiplication} that for any $\mathbf e \in K_0(B_T\modg)$ the map $\lambda_{\mathbf e} : f \mapsto \chi(\Gr_{\mathbf e}(F_T \cone(f)))$ is constructible. Thus, there exists a non-empty open subset $U_{\mathbf e}$ in $\Hom_{\CC}(T_1,T_0)$ such that $\lambda_{\mathbf e}$ is constant on $U_{\mathbf e}$. Since $\Hom_{\CC}(T_1,T_0)$ is a finite dimensional vector space, it is an irreducible affine variety and thus $U_{\mathbf e}$ is dense in $\Hom_{\CC}(T_1,T_0)$. Set $U_{\Gr} = \bigcap_{\mathbf e} U_{\mathbf e}$ where $\mathbf e$ runs over the (finite) set of elements in $K_0(B_T\modg)$ such that $\Gr_{\mathbf e}(F_T\cone(f))$ is non-empty. Then $U_{\Gr}$ is open and dense in $\Hom_{\CC}(T_1,T_0)$. 

			Also, since $\mathcal C$ has constructible cones with respect to $\mathcal T$-morphisms, it follows from \cite[Corollary 8]{Palu:multiplication} that $\iota:f \mapsto (\<S_i,F_T \cone(f)\>)_{i=1, \ldots n}$ is constructible. Thus, there exists a dense open subset $U_\iota \subset \Hom_{\CC}(T_0,T_1)$ such that $\iota$ is constant on $U_\iota$. 

			We finally set $U_{(T_1,T_0)} = U_\iota \cap U_{\Gr}$. It is a dense open subset in $\Hom_{\CC}(T_1,T_0)$ and the map $X: f \mapsto X(f)$ is constant over $U_{(T_1,T_0)}$, which proves the existence of a generic value. Unicity follows from the irreducibility of $\Hom_{\CC}(T_1,T_0)$.
		\end{proof}

	\subsection{Stabilisation maps and cluster characters}
		Let $T_0,T_1,T_0',T_1'$ be objects in $\mathcal T$. Then $[T_0]-[T_1]=[T_0']-[T_1']$ in $K_0(\mathcal T)$ if and only if there exists $T,T' \in \mathcal T$ such that $T_0 \oplus T= T_0' \oplus T'$ and $T_1 \oplus T= T_1' \oplus T'$. We define on $\mathcal T$ a structure of right-filter by setting $T' \leq T''$ if and only if there exists $T^{(3)} \in \mathcal T$ such that $T''=T' \oplus T^{(3)}$. 

		\begin{defi}
			Let $T_0,T_1$ be in $\mathcal T$ and $T' \leq T''$ be in $\mathcal T$ as above. The \emph{stabilisation map} from $\Hom_{\CC}(T_1 \oplus T',T_0 \oplus T')$ to $\Hom_{\CC}(T_1 \oplus T'',T_0 \oplus T'')$ is 
			$$\St_{T',T''}:\left\{\begin{array}{rcl}
				\Hom_{\CC}(T_1 \oplus T',T_0 \oplus T') & \fl & \Hom_{\CC}(T_1 \oplus T'',T_0 \oplus T'')\\
				f & \mapsto & f \oplus \id_{T^{(3)}}.
			\end{array}\right.$$
			We set $\St_{T'}=\St_{0,T'}$ the stabilisation map $\Hom_{\CC}(T_1,T_0) \fl \Hom_{\CC}(T_1 \oplus T', T_0 \oplus T')$ sending $f$ to $f \oplus \id_{T'}$.
		\end{defi}
		
		Given an element $\gamma \in K_0(\mathcal T)$, there exists a unique pair $(T_0^{\min}(\gamma), T_1^{\min}(\gamma))$ of objects in $\mathcal T$ such that 
		$$\gamma=[T_0^{\min}(\gamma)]-[T_1^{\min}(\gamma)]$$ and such that $T_0^{\min}(\gamma)$ and $T_1^{\min}(\gamma)$ have no common direct factor. 

		\begin{rmq}
			Let $T',T_0,T_1$ be objects in $\mathcal T$ and $f \in \Hom_{\CC}(T_1,T_0)$. The triangles 
			$$
				T_1 \xrightarrow{f} T_0 \fl \cone(f) \fl T_1[1] 
				\textrm{ and }
				T' \xrightarrow{\id_{T'}} T' \fl 0 \fl T'[1]
			$$
			give rise to
			$$T_1 \oplus T' \xrightarrow{f \oplus \id_{T'} } T_0\oplus T' \fl \cone(f) \fl T_1[1] \oplus T'[1]$$
			so that $\cone(f \oplus \id_{T'}) \simeq \cone(f)$ and thus $X(f \oplus \id_{T'})=X(f)$. Thus, $X$ is invariant under stabilisation. 

			Note that if $f$ is generic in $\Hom_{\CC}(T_1,T_0)$, the image of its $\Aut(T_0)^{\op} \times \Aut(T_1)$-orbit under a stabilisation map is a Zariski dense open subset in the image of $\Hom_{\CC}(T_1,T_0)$ under this stabilisation map. 
		\end{rmq}

		\begin{theorem}[Stability Theorem]\label{theorem:stability}
			Let $T_0, T_1$ be objects in $\mathcal T$ and $\gamma=[T_0]-[T_1] \in K_0(\mathcal T) $. Then a generic morphism in $\Hom_{\CC}(T_1,T_0)$ is isomorphic to an element in the image of the stabilisation map $\Hom_{\CC}(T_1^{\min}(\gamma),T_0^{\min}(\gamma)) \fl \Hom_{\CC}(T_1,T_0)$.
		\end{theorem}
		\begin{proof}
			The proof is the same as \cite[Theorem 5.2.2]{IOTW}. We recall it for completeness. Let $T_0, T_1$ be objects in $\mathcal T$ and $\gamma=[T_0]-[T_1] \in K_0(\mathcal T) $. Then, there exists some $T \in \mathcal T$ such that $T_0=T_0^{\min}(\gamma) \oplus T$ and $T_1=T_1^{\min}(\gamma) \oplus T$. Let $\phi \in \Hom_{\CC}(T_1,T_0)$ be a generic element. We now prove that there exists $(g_0,g_1) \in \Aut(T_0 \oplus T)^{\op} \times \Aut(T_1 \oplus T)$ such that $\phi=(g_0,g_1).\St_T(\phi^{\min})$ where $\phi^{\min} \in \Hom_{\CC}(T_1^{\min}(\gamma),T_0^{\min}(\gamma))$. 

			The element $\phi \in \Hom_{\CC}(T_1,T_0)$ can be viewed as a $2 \times 2$ matrix 
			$$\phi=\left[\begin{array}{cc}
				f & h \\
				g & r 
			\end{array}\right]$$
			with $f \in \Hom_{\CC}(T_1^{\min}(\gamma),T_0^{\min}(\gamma))$, $g \in \Hom_{\CC}(T_1^{\min}(\gamma),T)$, $h \in \Hom_{\CC}(T,T_0^{\min}(\gamma))$ and $r \in \End_{\CC}(T)$. Since $\phi$ is generic, it has maximal rank and in particular $r$ has full rank so that it is invertible. Thus, we get
			$$\phi=\left[\begin{array}{cc}
				f & h \\
				g & r 
			\end{array}\right]
			=
			\left[\begin{array}{cc}
				1_{T_0} & hr^{-1} \\
				0 & 1_{T} 
			\end{array}\right]
			\left[\begin{array}{cc}
				f-hr^{-1}g & 0 \\
				0 & 1_{T} 
			\end{array}\right]
			\left[\begin{array}{cc}
				1_{T_1} & 0\\
				g & r 
			\end{array}\right]
			$$
			so that 
			$$\phi=(g_0,g_1)\left[\begin{array}{cc}
				f-hr^{-1}g & 0 \\
				0 & 1_{T} 
			\end{array}\right]
			=
			(g_0,g_1)\St_T(\phi^{\min})$$
			where $$g_0=\left[\begin{array}{cc}
				1_{T_0} & hr^{-1}\\
				0 & 1_{T} 
			\end{array}\right] \in \Aut_{\CC}(T_0 \oplus T)^{\op},$$
			$$g_1=\left[\begin{array}{cc}
				1_{T_1} & 0\\
				g & r
			\end{array}\right] \in \Aut_{\CC}(T_1 \oplus T)$$
			and $\phi^{\min}=f-hr^{-1}g \in \Hom_{\CC}(T_1^{\min}(\gamma),T_0^{\min}(\gamma))$. This proves the theorem.
		\end{proof}

		\begin{corol}\label{corol:K0param}
			For any $T_0,T_1 \in \mathcal T$, the Laurent polynomial $X(T_1,T_0)$ only depends on $[T_0]-[T_1] \in K_0(\mathcal T)$.
		\end{corol}
		\begin{proof}
			Let $T_0,T_1$ be objects in $\mathcal T$ and set $\gamma =[T_0]-[T_1]$. It follows from the stability theorem that $X(T_1,T_0)=X(\St_T(\phi^{\min}))$ where $\phi^{\min}$ is a generic element in $\Hom_{\CC}(T_1^{\min}(\gamma),T_0^{\min}(\gamma))$. In particular, $X(T_1,T_0)=X(T_1^{\min}(\gamma),T_0^{\min}(\gamma))$ only depends on $\gamma$. 
		\end{proof}

		\begin{defi}
			For any $\gamma \in K_0(\mathcal T)$, a \emph{generic morphism of index $\gamma$} is a morphism in the dense open subset $U_{(T_1^{\min}(\gamma),T_0^{\min}(\gamma))} \subset \Hom_{\CC}(T_1^{\min}(\gamma),T_0^{\min}(\gamma))$.
		\end{defi}

		If $T_0, T_1$ are objects in $\mathcal T$ such that $\gamma=[T_0]-[T_1] \in K_0(\mathcal T)$, we will sometimes abuse terminology and view a generic morphism of index $\gamma$ as an element of $\Hom_{\CC}(T_1,T_0)$ by considering its image under the stabilisation map.

		We now define the generic characters as the images under $X$ of the generic morphisms.
		\begin{defi}
			For any $\gamma \in K_0(\mathcal T) $, the \emph{generic character of index $\gamma$ is}
			$$X(\gamma)=X(T_0^{\min}(\gamma),T_0^{\min}(\gamma))$$
			defined in Lemma \ref{lem:Xgen}.
			The set 
			$$\mathcal G^{T}(\CC)=\ens{X(\gamma) | \gamma \in K_0(\mathcal T) }$$
			is called the set of \emph{generic characters} associated to $T$ in $\CC$.
		\end{defi}

	\subsection{Generic morphism and general presentations in $B_{T}\modg$}\label{section:generalpresentations}
		The classical generic variables introduced in \cite{Dupont:BaseAaffine} for an acyclic quiver $Q$ were closely related to the generic representation theory of the finite dimensional hereditary $\k$-algebra $\k Q$ developed in \cite{Kac:infroot1,Kac:infroot2,Schofield:generalrepresentations}. In general, a 2-Calabi-Yau tilted algebra $B_{T}$ is a basic finite dimensional $\k$-algebra but is not hereditary. In \cite{DF:generalpresentations}, the authors develop a generic representation theory for any finite dimensional basic $\k$-algebra. This is done by replacing the usual notion of generic representation by the notion of generic presentation. We recall briefly this notion. 

		Given a finite dimensional basic $\k$-algebra $B$, we denote by $\epsilon_1, \ldots, \epsilon_n$ a maximal set of primitive idempotents of $B$. The indecomposable projective $B$-modules are $_BP_i=B\epsilon_i$ for $i \in \ens{1, \ldots, n}$. We denote by $K_0(B\projg)$ the Grothendieck group of the additive category $B{\projg}=\add(_BP_1 \oplus \cdots \oplus {_BP_n})$ and for any $B$-module $M$, we denote by $[M]$ its class in $K_0(B\projg)$. 

		For any projective $B$-modules $M_1,M_0$ and any morphism $f \in \Hom_{B}(M_1,M_0)$, the \emph{$\delta$-vector} of $f$ is $[M_0]-[M_1]$. Note that, identifying $K_0(B\projg)$ with $\Z^n$ by sending each $[_BP_i]$ to the $i$-th vector $\alpha_i$ of the canonical basis of $\Z^n$, the above definition coincides with the one provided in \cite{DF:generalpresentations}. 

		The group $\Aut_B(M_0)^{\op} \times \Aut_B(M_1)$ acts on $\Hom_B(M_1,M_0)$ by $(g_0,g_1).f=g_0fg_1$. A morphism $f \in \Hom_B(M_1,M_0)$ is called \emph{generic} if its $\Aut_B(M_0)^{\op} \times \Aut_B(M_1)$-orbit is a Zariski dense open subset in $\Hom_B(M_1,M_0)$. Such a generic morphism is called a \emph{general presentation} in $B$-mod. 

		\begin{rmq}
			As it was observed in \cite[Example 6.4]{DF:generalpresentations}, a minimal presentation of a $B$-module $M$ may not be a general presentation. Also, not every $B$-module admits a projective presentation which is a general presentation. A counterexample will for instance be provided in Remark \ref{rmq:cexgeneralpres}. 
		\end{rmq}

		We now prove that, under the functor $F_{T}:\CC \fl B_{T}\modg$, generic $\mathcal T$-morphisms in the category $\CC$ correspond to general presentations in $B_{T}\modg$.
		\begin{lem}\label{lem:genericpresentation}
			For any $T_0,T_1$ in $\mathcal T$ and any $f \in \Hom_{\CC}(T_1,T_0)$, the following are equivalent~:
			\begin{enumerate}
				\item $f$ is a generic $\mathcal T$-morphism in $\Hom_{\CC}(T_1,T_0)$~;
				\item $F_{T}f$ is a general presentation in $\Hom_{B_{T}}(F_{T}T_1,F_{T}T_0).$
			\end{enumerate}
		\end{lem}
		\begin{proof}
			We recall that the functor $F_{T}=\Hom_{\CC}(T,-)$ induces a $\k$-linear equivalence of categories $F_{T} : \CC/\mathcal T[1] \xrightarrow{\sim} B_{T}\modg$. Let $T_0,T_1$ be objects in $\mathcal T$. Since $T$ is a cluster-tilting object, for any $X$ in $\mathcal T[1]$, $\Hom_{\CC}(T_1,X)=0$ so that there are no morphisms from $T_1$ to $T_0$ in $\CC$ factorising through $\mathcal T[1]$. In particular, $F_{T}$ induces an isomorphism of $\k$-vector spaces 
			$$\Hom_{\CC}(T_1,T_0) \simeq \Hom_{B_{T}}(F_{T}T_1,F_{T}T_0)$$
			and this isomorphism is compatible with the actions of both $\Aut_{\CC}(T_0)^{\op} \times \Aut_{\CC}(T_1)$ and $\Aut_B(F_T T_0)^{\op} \times \Aut_B(F_T T_1)$. Thus, $f$ is generic in $\Hom_{\CC}(T_1,T_0)$ if and only if $F_{T}f$ is generic $\Hom_{B_{T}}(F_{T}T_1,F_{T}T_0)$
		\end{proof}

		\begin{corol}\label{corol:presentation}
			A presentation of $B_{T}$-modules is generic if and only if it is the image under $F_{T}$ of a generic $\mathcal T$-morphism. 
		\end{corol}
		\begin{proof}
			The result follows from Lemma \ref{lem:genericpresentation} and the fact that the projective $B_{T}$-modules are the $F_{T}M$ where $M$ runs over the non-zero objects in $\mathcal T$. 
		\end{proof}

		\begin{rmq}\label{rmq:deltavector}
			Note that for any $f \in \Hom_{\CC}(T_1,T_0)$, the $\delta$-vector of $F_{T}f$ is $[F_{T}T_0]-[F_{T}T_1] \in K_0(B\projg)$. Thus, using the above identifications of the Grothendieck groups $K_0(\mathcal T)$ and $K_0(B\projg)$ with $\Z^n$, we have~:
			$$\delta(F_{T}f)=[F_{T}T_0]-[F_{T}T_1]=[T_0]-[T_1]=\ind_{\mathcal T}(\cone(f)).$$ 
			Thus $F_{T}$ induces a 1-1 correspondence between generic $\mathcal T$-morphisms in $\CC$ and general presentations in $B_{T}$-mod and under this correspondence, a generic $\mathcal T$-morphism of index $\gamma$ corresponds to a general presentation of $\delta$-vector $\gamma$.
		\end{rmq}

\section{Cluster monomials and generic characters}\label{section:monomials}
	In this section, we assume that cluster-tilting subcategories of $\CC$ determine a cluster structure on $\CC$ (see Section \ref{ssection:clusterstructures}).

	We now prove that cluster monomials in $\mathcal A(Q_{T},\mathbf x)$ are generic cluster characters. We will see in Corollary \ref{corol:acyclic} that the converse is not true in general.
	\begin{theorem}\label{theorem:clustermonomials}
		Let $\CC$ be a triangulated 2-Calabi-Yau category with a cluster-tilting object $T$ and whose cones are constructible which respect to $\mathcal T$-morphisms. Assume that cluster-tilting subcategories determine a cluster structure on $\CC$. Then $$\mathcal M(Q_{T},\mathbf x) \subset \mathcal G^{T}(\CC).$$
	\end{theorem}
	\begin{proof}
		Let $M$ be a rigid object in $\CC$ and $T_1^M \xrightarrow{f_M} T_0^M \fl M \fl T_1^M[1]$ be a triangle in $\CC$ with $T_0^M,T_1^M \in \add T$ such that $T_0^M \fl M$ is a minimal right $\add T$-approximation. Set $\gamma=\ind_{\mathcal T}(M)=[T_0^M]-[T_1^M]$. Since $M$ is rigid in $\CC$, the objects $T_0^M$ and $T_1^M$ have no common direct summands \cite[Proposition 2.2]{DK:2CY} so that $T_0^M=T_0^{\min}(\gamma)$ and $T_1^M=T_1^{\min}(\gamma)$. Moreover, the $\Aut_{\CC}(T_0^M)^{\op} \times \Aut_{\CC}(T_1^M)$-orbit of $f_M$ is open and dense in $\Hom_{\CC}(T_1^M,T_0^M)$ \cite[\S 2.1]{DK:2CY}. It follows that $X(f_M)$ is the generic value of $X$ on $\Hom_{\CC}(T_1^M,T_0^M)$. Since $\cone(f_M) \simeq M$, it follows from the definition of $X$ that $X(f_M)=X^{T}_M$ and thus $X^{T}_M=X(\ind_{\mathcal T}(M)) \in \mathcal G^{T}(\CC)$. 	

		As cluster-tilting subcategories form a cluster structure on $\CC$, it follows from \cite{FK} that the cluster character $X^{T}_?$ induces a surjection from the set of reachable rigid objects to cluster monomials in $\mathcal A(Q_{T},\mathbf x)$. It thus follows from the above discussion that $\mathcal M(Q_{T},\mathbf x) \subset \mathcal G^{T}(\CC)$.
	\end{proof}

	We have actually proved the more precise statement~:
	\begin{prop}
		Let $\CC$ be a triangulated 2-Calabi-Yau category with a cluster-tilting object $T$ and whose cones are constructible which respect to $\mathcal T$-morphisms. Assume that cluster-tilting subcategories determine a cluster structure on $\CC$. Then the following hold~:
		\begin{enumerate}
			\item If $M$ is a rigid object in $\CC$, then $X^{T}_M=X(\ind_{\mathcal T}(M))$~;
			\item If $M$ is a rigid object in $\CC$ which is reachable from $T$, then $X(\ind_{\mathcal T}(M))$ is a cluster monomial in $\mathcal A(Q_{T},\mathbf x)$~;
			\item If $M$ is an indecomposable rigid object in $\CC$ which is reachable from $T$, then $X(\ind_{\mathcal T}(M))$ is a cluster variable in $\mathcal A(Q_{T},\mathbf x)$.
		\end{enumerate}
	\end{prop}
	\begin{proof}
		This follows immediately from the proof of Theorem \ref{theorem:clustermonomials}.
	\end{proof}

	We now prove that in finite type, the set of generic characters coincides with the set of cluster monomials~:
	\begin{theorem}\label{theorem:finitetype}
		Let $\CC$ be a triangulated 2-Calabi-Yau category with a cluster-tilting object $T$ and whose cones are constructible which respect to $\mathcal T$-morphisms. Assume that cluster-tilting subcategories determine a cluster structure on $\CC$. Assume moreover that $\CC$ has finitely many indecomposable objects up to isomorphisms. Then 
		$$\mathcal G^{T}(\CC) = \mathcal M(Q_{T},\mathbf x).$$
	\end{theorem}
	\begin{proof}
		The cluster character $X^T_?$ induces a surjection from the set of reachable indecomposable rigid objects in $\CC$ to the set of cluster variables in $\mathcal A(Q_T,\mathbf x)$. If $\CC$ has finitely many indecomposable objects, it follows that $\mathcal A(Q_T,\mathbf x)$ has finitely many cluster variables and thus $Q_T$ is mutation-equivalent to a Dynkin quiver $Q$. It thus follows from \cite{KR:acyclic} that $\CC$ is triangle-equivalent to the cluster category $\CC_Q$ of $Q$. In this case, for any $\gamma \in K_0(\add T)$ there exists a rigid object $M_\gamma$ in $\CC$ such that $\ind_{\mathcal T}(M_\gamma) = \gamma$. Thus, as in the proof of Theorem \ref{theorem:clustermonomials}, $X(\gamma) = X(f_{M_\gamma})=X^T_{M_\gamma}$ where $f_{M_\gamma}$ is a minimal right-$\add T$-approximation of $M_\gamma$ and thus $X(\gamma) \in \mathcal M(Q_T,\mathbf x)$.
	\end{proof}

	\begin{corol}\label{corol:basetypefini}
		Let $\CC$ be a triangulated 2-Calabi-Yau category with a cluster-tilting object $T$ and whose cones are constructible which respect to $\mathcal T$-morphisms. Assume that cluster-tilting subcategories determine a cluster structure on $\CC$. Assume moreover that $\CC$ has finitely many indecomposable objects up to isomorphisms. Then $\mathcal G^{T}(\CC)$ is a $\Z$-basis in $\mathcal A(Q_T,\mathbf x)$.
	\end{corol}
	\begin{proof}
		We saw in the proof of Theorem \ref{theorem:finitetype} that under these hypotheses $\mathcal A(Q_T,\mathbf x)$ is of finite type. It thus follows from \cite[Corollary 3]{CK1} that $\mathcal M(Q_T,\mathbf x)$ is a $\Z$-linear basis of $\mathcal A(Q_T,\mathbf x)$. The corollary is hence a consequence of Theorem \ref{theorem:finitetype}.
	\end{proof}

\section{Indices and dimension vectors in cluster categories}\label{section:index}
	From now on, we assume that $\CC=\CC_Q$ is the cluster category of an acyclic quiver $Q$ with vertices $\ens{1, \ldots, n}$ and that $T=\kQ$ is the canonical cluster-tilting object in $\CC_Q$. We recall that the cluster category is the orbit category in $D^b(\kQ\modg)$ of the functor $\tau^{-1}[1]$ where $[1]$ is the suspension functor in the bounded derived category $D^b(\kQ\modg)$ and $\tau$ is the Auslander-Reiten translation. It is a canonically triangulated 2-Calabi-Yau category \cite{K} with constructible cones with respect to $\add T$-morphisms \cite{Palu:multiplication}.

	For any $i \in Q_0$, we denote by $S_i$ the simple $\kQ$-module at vertex $i$ and by $P_i$ its projective cover. The set of isomorphism classes of indecomposable objects in $\CC_Q$ can be identified with the disjoint union of the set of isomorphism classes of indecomposable $\kQ$-modules and of shifts of indecomposable projective modules. The canonical cluster-tilting object $T$ can thus be written as $T=\kQ=\bigoplus_{i=1}^nP_i$.

	\subsection{Dimension vectors and indices}
		We shall now compare the notion of index in the cluster category $\CC_Q$ to the notion of dimension vector in the module category $\kQ$-mod.

		For any $\kQ$-module $M$, the \emph{dimension vector} of $M$ is the element $\ddim M=(\dimk \Hom_{\kQ}(P_i,M))_{1 \leq i \leq n} \in \Z_{\geq 0}^n$. Let $K_0(\kQ\modg)$ denote the Grothendieck group of $\k Q$-mod.  It is known that $\ddim$ induces an isomorphism of abelian groups $K_0(kQ\modg) \xrightarrow{\sim} \Z^n$ sending the isoclass of the simple $S_i$ to the $i$-th vector $\alpha_i$ of the canonical basis of $\Z^n$ for any $i \in Q_0$.

		Since $Q$ is acyclic, the path algebra $\kQ$ is finite dimensional and hereditary so that the bilinear form $\<-,-\>$ is well defined on $\Z^n \simeq K_0(\kQ\modg)$. The \emph{Euler matrix} of $Q$ is thus the matrix $E \in M_n(\Z)$ of the (non-symmetric) bilinear form $\<-,-\>$ on $K_0(\kQ\modg)$. We refer the reader to \cite[\S III.3]{ASS} for classical properties of $E$.

		As usual, we identify $\kQ$-mod with the category $\rep(Q)$ of finite dimensional representations of $Q$ over $\k$. We recall that a \emph{representation} $M$ of $Q$ is a pair $M=((M(i))_{i \in Q_0},(M(\alpha))_{\alpha \in Q_1})$ such that each $M(i)$ is a finite dimensional $\k$-vector space and each $M(\alpha)$ is a $\k$-linear map $M(i) \fl M(j)$ where $\alpha:i \fl j \in Q_1$. Note that $\ddim M=(\dimk M(i))_{i \in Q_0}$ for any representation $M$ of $Q$.

		For any $\mathbf d \in \Z_{\geq 0}^n$, we denote by $\rep(Q,\mathbf d)$ the set of representations $M$ of $Q$ such that $\ddim M=\mathbf d$ which is identified with the irreducible affine variety 
		$$\rep(Q,\mathbf d)=\prod_{\alpha: i \fl j \in Q_1} \Hom_{\k}(\k^{d_i},\k^{d_j}),$$
		called \emph{representation space of dimension $\mathbf d$}.

		We define the \emph{dimension vector} $\ddimC M$ of an object $M$ in the cluster category by
		$$\ddimC M =\left\{\begin{array}{ll}
			\ddim M & \textrm{ if $M$ is an indecomposable $\kQ$-module~;}\\
			- (E^t)^{-1} \ddim S_i & \textrm{ if } M \simeq P_i[1]~;\\
			\ddimC M_1 + \ddimC M_2 & \textrm{ if $M=M_1 \oplus M_2$.}
		\end{array}\right.$$
		Since $\ddimC M=\ddim M$ for any $\kQ$-module $M$, we will simply write $\ddim M$ for the dimension vector of an arbitrary object $M$ in $\CC$.

		\begin{rmq}
			Note that our convention for dimension vectors of objects in $\CC_Q$ agrees with the one considered in \cite{IOTW} but differs from the one considered \cite{CK2,Dupont:genericvariables}. Indeed, in \cite{CK2,Dupont:genericvariables}, the dimension vector of $P_i[1]$ was set to $-\ddim S_i$. This was more accurate from the point of view of denominator vectors of the corresponding cluster characters (see for instance \cite[Theorem 3]{CK2} and \cite[Proposition 4.1]{Dupont:genericvariables}). Nevertheless, as we shall see, it appears that the convention of \cite{IOTW} we use here is more natural from the point of view of indices, $\mathbf g$-vectors, virtual generic decompositions and generic characters.
		\end{rmq}

		\begin{lem}\label{lem:explicitindex}
			Let $\CC=\CC_Q$ be the cluster category of an acyclic quiver $Q$. Let $T=\k Q$ be the canonical cluster tilting object in $\CC$ and $\mathcal T=\add T$. Then for any $\kQ$-module $M$ in $\CC$, we have 
			$$\ind_{\mathcal T}(M)=E^t \ddim M,$$
			$$\coind_{\mathcal T}(M)=E \ddim M.$$
		\end{lem}
		\begin{proof}
			Without loss of generality, we can assume that $M$ is indecomposable. Then it follows from \cite[Lemma 2.3]{Palu} that $\ind_{\mathcal T}(M) = (\<M,S_i\>)_{i \in Q_0} = (\ddim M^t E \alpha_i)_{i \in Q_0} = E^t \ddim M$ and $\coind_{\mathcal T}(M) = (\<S_i,M\>)_{i \in Q_0} = (\alpha_i E \ddim M)_{i \in Q_0} = E \ddim M$.
		\end{proof}

		\begin{rmq}
			Note that if $M \simeq P_i[1]$, then $\ind_{\T}(M) = -[P_i] = -\alpha_i = E^t \ddim P_i[1]$ but $\coind_{\T}(M) = -[P_i] = - \alpha _i \neq E \ddim P_i[1]$ so that Lemma \ref{lem:explicitindex} only holds for $\kQ$-modules.
		\end{rmq}

	\subsection{Presentation spaces and index}
		Following the approach in \cite{IOTW}, we will consider presentation spaces in $\kQ$-mod instead of representations spaces in $\rep(Q)$. 

		Given an element $\alpha \in \Z^n$, a \emph{projective decomposition} of $\alpha$ is a pair $(\gamma_0,\gamma_1) \in \Z_{\geq 0}^n \times \Z_{\geq 0}^n$ such that $E^t \alpha=\gamma_0 - \gamma_1$. It is called \emph{minimal} if $\gamma_0$ and $\gamma_1$ have disjoint support. A minimal projective decomposition of a given element $\gamma \in \Z_{\geq 0}^n$ is unique. 

		For any $\gamma \in \Z_{\geq 0}^n$, we set $P(\gamma)=\bigoplus_{i \in Q_0} P_i^{ \gamma_i}$. For any $\alpha \in \Z^n$, the \emph{presentation space} associated to a projective decomposition $(\gamma_0,\gamma_1)$ of $\alpha$ is 
		$$R(\gamma_0,\gamma_1)=\Hom_{\k Q}(P(\gamma_1),P(\gamma_0)).$$
		The \emph{minimal presentation space} $R^{\min}(\alpha)$ is the representation space associated to the minimal projective decomposition of $\alpha$.

		Let $\alpha$ be an element in $\Z^n$, which shall be thought as a dimension vector of an object in $\CC_Q$. Then, by Lemma \ref{lem:explicitindex}, $\gamma = E^t \alpha \in \Z^n$ may be thought as the index of an object in $\CC_Q$. Let $(\gamma_0,\gamma_1)$ be a projective decomposition of $\alpha$ then $\gamma=\gamma_0 - \gamma_1=[P(\gamma_0)]-[P(\gamma_1)]$. The objects $T_0=P(\gamma_0)$ and $T_1=P(\gamma_1)$ belong to $\mathcal T=\add \kQ=\kQ\projg$ and the representation space $R(\gamma_0,\gamma_1)$ is thus isomorphic to $\Hom_{\kQ}(T_1,T_0)$ which is isomorphic to $\Hom_{\CC}(T_1,T_0)$ since $T_1$ is a projective $\kQ$-module \cite{BMRRT}.

		If $(\gamma^{\min}_0, \gamma^{\min}_1)$ is the minimal projective decomposition of $\alpha$, we have $T_0^{\min}(\gamma)=P(\gamma_0^{\min})$ and $T_1^{\min}(\gamma)=P(\gamma_1^{\min})$. Thus, the minimal presentation space $R^{\min}(\alpha)$ is isomorphic to the space of $\mathcal T$-morphisms $\Hom_{\CC}(T_1^{\min}(\gamma),T_0^{\min}(\gamma))$. In particular, the generic morphism of index $\gamma$ can be viewed as a generic element of the minimal presentation space $R^{\min}((E^t)^{-1}\gamma)$ .

	\subsection{$\mathbf g$-vectors}
		Generic variables introduced in \cite{Dupont:genericvariables} were naturally parametrised by their denominator vectors. We now prove that generic characters we just introduced are naturally parame\-trised by their $\mathbf g$-vectors. For details concerning $\mathbf g$-vectors, we refer the reader to \cite{cluster4} in general and \cite{FK} for $\mathbf g$-vectors in the context of cluster characters. 

		Given an object $M$ in $\CC$, we set $\mathbf g_M = -\coind_{\mathcal T}(M)$. Given $\gamma \in K_0(\add T)$, we define the \emph{$\mathbf g$-vector of $X(\gamma)$} as $\mathbf g_{X(\gamma)} = \mathbf g_{\cone(f)}$ for some generic morphism of index $\gamma$ (it follows from the proof of Lemma \ref{lem:Xgen} that this is well-defined). Usually, the $\mathbf g$-vector is only defined for cluster variables in cluster algebras with coefficients but our terminology is motivated by the fact that, under certain assumptions on the coefficient system of the cluster algebra $\mathcal A(Q_T,\mathbf x)$, if $M$ is a rigid object in $\CC$, then $\mathbf g_M$ is indeed the $\mathbf g$-vector of the cluster variable $X_M$ \cite[Theorem 6.3]{FK}. 

		We denote by $C=-E^tE^{-1}$ the \emph{Coxeter matrix} of the path algebra $\kQ$. 
		
		\begin{prop}\label{prop:gvector}
			Let $\CC$ be the cluster category of an acyclic quiver $Q$ and $T=\kQ$ be the canonical cluster-tilting object. Then for any $\gamma \in \Z^n$ such that the cone of a generic morphism of index $\gamma$ is a $\kQ$-module, we have
			$${\bf{g}}_{X(\gamma)}=C^{-1}\gamma.$$
		\end{prop}
		\begin{proof}
			Fix $\gamma \in \Z^n$ and let $M$ be the cone of a generic morphism of index $\gamma$ which is a $\kQ$-module by hypothesis. By Lemma \ref{lem:explicitindex}, we have $\gamma=\ind_{T}(M)=E^t \ddim M$. Now, by definition, $\mathbf g_{M}=-\coind_{T}(M)$ and thus, Lemma \ref{lem:explicitindex} implies that $-\coind_{T}(M)=-E \ddim M$. 
			It follows that $\gamma = E^t \ddim M =  E^t (-E^{-1}) \mathbf g_{M} = C \mathbf g_M$ so that $\mathbf g_{X(\gamma)}=\mathbf g_M=C^{-1}\gamma$.
		\end{proof}

		\begin{rmq}
			Note that for any $i \in Q_0$, we have $X(-\alpha_i) = X^T_{T_i[1]} = x_i$ so that $\mathbf g_{x_i} = -\alpha_i = -\alpha_i \neq C^{-1} \alpha_i$. Thus, Proposition \ref{prop:gvector} does not hold if the cone of a generic morphism of index $\gamma$ is not a $\kQ$-module.
		\end{rmq}

		\begin{corol}
			Let $M$ be a $\kQ$-module, and $P_0^M \xrightarrow{f_M} P_1^M \fl M \fl 0$ be a projective presentation of $\kQ$-modules. Then,
			$$\delta(f_M)=C\mathbf g_M.$$
		\end{corol}
		\begin{proof}
			It follows from Remark \ref{rmq:deltavector} that $\delta(f_M)=\ind_{T}(\cone(f_M))=\ind_{T}(M)=E^t \ddim M=C\mathbf g_M$.
		\end{proof}

\section{Generic cluster characters and virtual generic decomposition}\label{section:genericdcp}
	In this section, we still assume that $\CC=\CC_Q$ is the cluster category of an acyclic quiver $Q$ and that $T=\kQ$ is the canonical cluster-tilting object in $\CC$. 

	Following \cite{Schofield:generalrepresentations}, given two elements $\beta,\gamma \in \Z_{\geq 0}^n$, we say that the extension $\Ext^1_{\kQ}(\beta,\gamma)$ \emph{vanishes generally} if there exist Zariski dense open subsets $O_\beta \subset \rep(Q,\beta)$, $O_{\gamma} \subset \rep(Q,\gamma)$ such that $\Ext^1_{\kQ}(M_\beta,M_\gamma)=0$ for any $M_\beta \in O_\beta$ and $M_\gamma \in O_\gamma$. 

	Given an element $\alpha \in \Z_{\geq 0}^n$, Kac proved in \cite{Kac:infroot1} the existence of a unique decomposition
	$$\alpha=\beta_1 + \cdots + \beta_k$$
	such that~:
	\begin{enumerate}
		\item $\Ext^1_{\kQ}(\beta_i,\beta_j)$ vanishes generally if $i \neq j$~;
		\item each $\beta_i$ is a Schur root of $Q$.
	\end{enumerate}
	This decomposition is called the \emph{generic (or canonical) decomposition of $\alpha$}.

	This was generalised in \cite{IOTW} to arbitrary elements in $\Z^n$. Namely, for any $\alpha \in \Z^n$ there exists a unique decomposition 
	$$\alpha=\beta_1 + \cdots + \beta_k - (E^t)^{-1}\gamma$$
	such that~:
	\begin{enumerate}
		\item $\beta_1, \ldots, \beta_k,\gamma \in \Z_{\geq 0}^n$~;
		\item $\beta_i$ and $\gamma$ have disjoint support for all $i$~;
		\item $\Ext^1_{\kQ}(\beta_i,\beta_j)$ vanishes generally if $i \neq j$~;
		\item each $\beta_i$ is a Schur root of $Q$.
	\end{enumerate}
	This decomposition is called the \emph{virtual generic decomposition of $\alpha$}. Note that if $\alpha \in \Z_{\geq 0}^n$, generic and virtual generic decomposition coincide.
	
	\begin{theorem}\label{theorem:multvirtual}
		Let $\CC$ be the cluster category of an acyclic quiver $Q$ and $T=\kQ$ be the canonical cluster-tilting object in $\CC$. Let $\alpha \in \Z^n$ with virtual generic decomposition
		$$\alpha=\beta_1 + \cdots + \beta_k - (E^t)^{-1}\gamma.$$
		Then,
		$$X(E^t\alpha)=X(E^t \beta_1) \cdots X(E^t\beta_k)X(-\gamma).$$
	\end{theorem}
	\begin{proof}
		Let $\alpha \in \Z^n$ with virtual canonical decomposition $\alpha=\beta_1 + \cdots + \beta_k - (E^t)^{-1}\gamma$. For any $i=1, \ldots, k$, let $M_i$ be a generic representation in $\rep(Q,\beta_i)$. For any $i=1, \ldots, k$, we denote by $p_{M_i}$ the so-called \emph{canonical projective presentation} of $M_i$:
		$$0 \fl P^{M_i}_1 \xrightarrow{p_{M_i}} P^{M_i}_0 \fl M \fl 0$$
		where $$P^{M_i}_1 = \bigoplus_{u \fl v \in Q_1} P_v^{ \beta_i(u)} \textrm{ and } P^{M_i}_0 = \bigoplus_{v \in Q_0} P_v^{ \beta_i(v)}$$
		and let $0_{\gamma}$ be the zero map $0_{\gamma}:P(\gamma) \fl 0$. Then, it follows from \cite[Theorem 6.3.1]{IOTW} that $$X(E^t\alpha) = X\left(\bigoplus_{i=1}^k p_{M_i} \oplus 0_{\gamma}\right) = \prod_{i=1}^k X(p_{M_i}) X(0_{\gamma}).$$
		Each $p_{M_i}$ is generic so that $X(p_{M_i})=X(\ind_{T}(M_i))=X(E^t \ddim M_i)=X(E^t \beta_i)$ and also, $0_{\gamma}$ is generic so that $X(0_\gamma)=X(\ind_{T}(P_{\gamma}[1]))=X(-\gamma)$. This finishes the proof.
	\end{proof}

\section{Generic cluster characters and classical generic variables}\label{section:XetCC}
	In this section, $\CC$ still denotes the cluster category of an acyclic quiver $Q$ and $T=\kQ$ is the canonical cluster-tilting object in $\CC$. In this case, the cluster character $X^{T}_?$ coincides with the Caldero-Chapoton map $CC:\Ob(\CC) \fl \Z[\mathbf x^{\pm 1}]$ introduced in \cite{CC,CK2}. 
	In \cite{Dupont:genericvariables}, the author introduced a family of Laurent polynomials in $\Z[\mathbf x^{\pm 1}]$, called \emph{generic variables}, by considering generic values of the restriction of the Caldero-Chapoton map to representation spaces. We shall now see that generic characters coincide with these generic variables when $\CC=\CC_Q$ and $T=\kQ$. First, we briefly review the construction of \cite{Dupont:genericvariables}.

	\subsection{Generic variables in acyclic cluster algebras}
		For any $\alpha \in \Z_{\geq 0}^n$, there exists a unique Laurent polynomial $CC(\alpha)$ such that $CC$ is constant equal to $CC(\alpha)$ on a dense open subset $U_{\alpha} \subset \rep(Q,\alpha)$ \cite[Corollary 2.4]{Dupont:genericvariables}.

		More generally, if $\alpha \in \Z^n$, let $\alpha_{0},\alpha_1 \in \Z_{\geq 0}^n$ having disjoint support such that $\alpha=\alpha_0 - \alpha_1$ and set
		$$CC(\alpha)=CC(\alpha_0)\mathbf x^{\alpha_1}=CC(\alpha_0)CC(P(\alpha_1)[1]).$$
		The set 
		$$\mathcal G(Q)=\ens{CC(\alpha)\, |\, \alpha \in \Z^n}$$
		is called the set of \emph{generic variables} in $\mathcal A(Q)$ and for any $\alpha \in \Z^n$, $CC(\alpha)$ is called the \emph{generic variable of dimension $\alpha$}. 

		Let $\alpha, \beta \in \Z^n$ such that $\alpha=\alpha_0 - \alpha_1$ (resp. $\beta=\beta_0 - \beta_1$) where $\alpha_1,\alpha_0$ (resp. $\beta_1,\beta_0$) have disjoint support. Following \cite{Dupont:genericvariables}, we say that $\Ext^1_{\CC}(\alpha,\beta)$ \emph{vanishes generally} if there exists $M_\alpha \in \rep(Q,\alpha_0), M_\beta \in \rep(Q,\beta_0)$ such that 
		$$\Ext^1_{\CC}(M_\alpha \oplus P(\alpha_1)[1],M_\beta \oplus P(\beta_1)[1])=0.$$
		Generic variables are multiplicative if there are no generic extensions in the cluster category in the sense that $$CC(\alpha+\beta)=CC(\alpha)CC(\beta).$$
		if $\Ext^1_{\CC}(\alpha,\beta)$ vanishes generally \cite[Lemma 3.5]{Dupont:genericvariables}. In particular, they are compatible with Kac's generic decomposition in the sense that for $\alpha \in \Z_{\geq 0}^n$, if
		$$\alpha=\beta_1 + \cdots + \beta_k$$
		is the generic decomposition of $\alpha$, then
		$$CC(\alpha)=CC(\beta_1) \cdots CC(\beta_k).$$

		We recall that if $L$ is a Laurent polynomial in $\Z[\mathbf x^{\pm 1}]$, its \emph{denominator vector} is the unique vector $\mathbf d \in \Z^n$ such that there exists a polynomial $P(x_1, \ldots, x_n)$ not divisible by any $x_i$ such that $L=P(x_1, \ldots, x_n)/\mathbf x^{\mathbf d}$. Generic variables are parametrised by denominator vectors in the sense that for $\alpha \in \Z^n$, the denominator vector of $CC(\alpha)$ is $\alpha$. 

	\subsection{Generic variables and generic characters}
		We now prove that the set of generic variables coincides with the set of generic characters. The philosophy underlying the proof of this fact is that these two sets are naturally parametrised by $\Z^n$. The set $\mathcal G(Q)$ of generic variables is naturally parametrised by denominator vectors which essentially correspond to dimension vectors of $\kQ$-modules. The set $\mathcal G^{\kQ}(\CC_Q)$ is parametrised using the index in the corresponding cluster category. Going from one parametrisation to the other correspond to the base change induced by the matrix $E^t$ between the two natural bases of the lattice $\Z^n$, the first one consisting of dimension vectors of simple modules, the second one of dimension vectors of indecomposable projective modules. Nevertheless, due to the fact that we do not use the same convention for dimension vectors of shifts of indecomposable modules as in \cite{Dupont:genericvariables}, the change of parametrisation between $\mathcal G(Q)$ and $\mathcal G^{\kQ}(\CC_Q)$ will be slightly more complicated.

		\begin{lem}\label{lem:XetCCpositif}
			For any $\alpha \in \Z_{\geq 0}^n$, 
			$$X(E^t\alpha)=CC(\alpha).$$
		\end{lem}
		\begin{proof}
			Let $\alpha \in \Z_{\geq 0}^n$. For any projective decomposition $E^t\alpha=\gamma_0-\gamma_1$, $X(E^t\alpha)=X(f)$ for some generic element $f \in \Hom_{kQ}(P(\gamma_1),P(\gamma_0))$. We can thus assume that $E^t\alpha=\gamma_0-\gamma_1$ is the projective canonical decomposition of \cite{IOTW}. Consider the triangle
			$$P(\gamma_1) \xrightarrow{f} P(\gamma_0) \fl \cone(f) \fl P(\gamma_1)[1].$$
			Applying $F_{T}$, we get 
			$$P(\gamma_1) \xrightarrow{F_{T}f} P(\gamma_0) \fl F_{T}(\cone(f)) \fl 0$$
			so that $X(f)=X^{T}_{\cone(F_{T}(f))}=X^{T}_{\coker F_{T}(f)}.$ According to Corollary \ref{corol:presentation}, $F_{T}f$ is generic in $\Hom_{kQ}(P(\gamma_1),P(\gamma_0))$. Thus, we get 
			$$X(f)=X^{T}(\cone f)=CC(\coker F_{T} f).$$
			Consider the map $\kappa: g \mapsto \coker g$ on $\Hom_{kQ}(P(\gamma_1),P(\gamma_0))$. For any morphism $g \in \Hom_{kQ}(P(\gamma_1),P(\gamma_0))$, we have $\ddim \coker(g)=\ddim P(\gamma_0) - \ddim P(\gamma_1)=(E^t)^{-1}(\gamma_0 - \gamma_1)=\alpha$ so that $\kappa$ is a map 
			$$\kappa:\Hom_{kQ}(P(\gamma_1),P(\gamma_0)) \fl \rep(Q,\alpha).$$
			Now, it follows from \cite[Proposition 4.1.7]{IOTW} that there exists a Zariski dense open subset $\mathcal U \subset \Hom_{kQ}(P(\gamma_1),P(\gamma_0))$ such that $\kappa_{|\mathcal U}:\mathcal U \fl \rep(Q,\alpha)$ is algebraic and injective. Thus, $\dim \kappa(\mathcal U) \geq \dim \mathcal U > 0$ and thus $\kappa(\mathcal U) \cap U_{\alpha} \neq \emptyset$ where $U_{\alpha}$ denotes, as before, the Zariski dense open subset for the Caldero-Chapoton map.
			It follows that $\kappa^{-1}(U_{\alpha})$ is a Zariski dense open subset in $\Hom_{kQ}(P(\gamma_1),P(\gamma_0))$ and we can thus assume that the generic element $f$ belongs to $\kappa^{-1}(U_{\alpha})$. It follows that 
			$$X(E^t\alpha)=X(f)=CC(\coker F_{T}f)=CC(\alpha)$$
			which proves the lemma.
		\end{proof}

		\begin{lem}\label{lem:XetCCnegatif}
			For any $\alpha \in \Z_{\leq 0}^n$, 
			$$X(\alpha)=CC(\alpha).$$
		\end{lem}
		\begin{proof}
			Let $\gamma \in \Z_{\geq 0}^n$ such that $\alpha=-\gamma$. Consider the morphism $0_{\gamma}: \Hom_{kQ}(P(\gamma),0)$. We have the triangle 
			$$P(\gamma) \xrightarrow{0_{\gamma}} 0 \fl P(\gamma)[1] \xrightarrow{\sim} P(\gamma)[1]$$
			so that $$X(0_{\gamma})=X^{T}_{P(\gamma)[1]}=CC(P(\gamma)[1])=CC(-\gamma)=CC(\alpha).$$
			Now, by definition, $X(f)$ is the generic character of index $\alpha$, that is, $X(\alpha)=X(0_{\gamma})=CC(\alpha)$.
		\end{proof}

		\begin{prop}\label{prop:XetCC}
			Let $\CC=\CC_Q$ be the cluster category of an acyclic quiver $Q$ and $T=\kQ$ be the canonical cluster-tilting object in $\CC$. Let $\alpha \in \Z^n$ with virtual generic decomposition
			$$\alpha=\beta_1 + \cdots + \beta_k - (E^t)^{-1}\gamma.$$
			Then,
			$$X(E^t\alpha)=CC(\beta_1 + \cdots + \beta_k - \gamma).$$
		\end{prop}
		\begin{proof}
			According to Theorem \ref{theorem:multvirtual}, we have 
			$$X(E^t\alpha)=X(E^t\beta_1)\cdots X(E^t\beta_k)X(-\gamma).$$
			For any $i=1, \ldots, k$, it follows from Lemma \ref{lem:XetCCnegatif} that $X(E^t\beta_i)=CC(\beta_i)$. Also, it follows from Lemma \ref{lem:XetCCnegatif} that $X(-\gamma)=CC(-\gamma)$. Now, $\beta_1 + \cdots + \beta_k$ is the (classical) generic decomposition of $\sum_{i=1}^k \beta_i$
			so that 
			$$CC(\beta_1)\cdots CC(\beta_k)=CC(\beta_1 + \cdots + \beta_k).$$
			Moreover, $\beta_i$ and $\gamma$ have disjoint support for any $i \in \ens{1, \ldots, k}$ so that $\gamma$ and $\sum_{i=1}^k \beta_i$ have disjoint support and $\Ext^1_{\CC}(\sum_{i=1}^k \beta_i,-\gamma)$ vanishes generally. 
			Thus, $$CC(\sum_{i=1}^k \beta_i)CC(-\gamma)=CC(\sum_{i=1}^k \beta_i-\gamma)$$
			which proves the proposition.
		\end{proof}

		\begin{theorem}\label{theorem:generalisation}
			Let $\CC=\CC_Q$ be the cluster category of an acyclic quiver $Q$ and $T=\kQ$ be the canonical cluster-tilting object in $\CC$. Then, 
			$$\mathcal G(Q)=\mathcal G^{T}(\CC).$$
		\end{theorem}
		\begin{proof}
			$E^t$ is an invertible matrix so $E^t\alpha$ runs over $\Z^n$ when $\alpha$ runs over $\Z^n$. It thus follows from Proposition \ref{prop:XetCC} that $\mathcal G(Q) \supset \mathcal G^{T}(\CC)$. 

			Conversely, if $\alpha \in \Z^n$, we can write $\alpha=\alpha_0 - \alpha_1$ where $\alpha_0,\alpha_1 \in \Z_{\geq 0}^n$ have disjoint supports. Let $M$ be the generic representation in $\rep(Q,\alpha_0)$ and let $p_M$ be its canonical projective presentation. Then $CC(\alpha_0) = CC(M) = X(p_M) = X(E^t\alpha_0)$. Let $0_{\alpha_1}$ be the zero morphism $P(\alpha_1) \fl 0$. Then 
			$$X(E^t\alpha_0-\alpha_1) = X(p_M \oplus 0_{\alpha_1}) = X(p_M)X(0_\alpha) = CC(C)CC(P(\alpha_1)[1]) = CC(\alpha)$$
			so that $\mathcal G^T(\CC) \subset \mathcal G(Q)$. This proves the theorem.
		\end{proof}

		\begin{rmq}\label{rmq:cexgeneralpres}
			We now observe that there may exist indecomposable modules over finite-dimensional algebras which admit no generic projective presentations. Consider for instance the quiver 
			$$\xymatrix@-3ex{
				&& 2 \ar[rd] \\
				Q: & 1 \ar[ru] \ar[rr] && 3
			}$$
			and the indecomposable (non-rigid) representation
			$$\xymatrix@-3ex{
				&& \k \ar[rd]^0 \\
				M: & \k \ar[ru]^1 \ar[rr]_1 && \k.
			}$$
			Let $\CC$ be the cluster category of $Q$ and $T$ be the canonical cluster-tilting object in $\CC$. If $M$ admits a generic projective presentation $P_1^M \xrightarrow{f_M} P_0^M$ in $\kQ$-mod, then it follows from Lemma \ref{lem:genericpresentation} that $f_M$ induces a generic morphism in $\Hom_{\CC}(P_1^M,P_0^M)$ and thus $X_M = X^{T}_{\cone(f_M)} = X(f_M) \in \mathcal G^{T}(\CC)$ but it is known that $X_M \not \in \mathcal G(Q)$ (see the proof of \cite[Lemma 5.4]{Dupont:genericvariables} for details on this last fact).
		\end{rmq}

	\subsection{Generic bases in acyclic cluster algebras}
		Combining with known results on classical generic variables, we get~:
		\begin{corol}\label{corol:acyclic}
			Let $\CC=\CC_Q$ be the cluster category of an acyclic quiver $Q$ and $T=\kQ$ be the canonical cluster-tilting object in $\CC$. Then $\mathcal G^{T}(\CC)$ is a $\Z$-linear basis in $\mathcal A(Q,\mathbf x)$.
		\end{corol}
		\begin{proof}
			Let $Q$ be an acyclic quiver, $\CC=\CC_Q$ and $T=\kQ$. In \cite{GLS:generic} (see also \cite{Dupont:BaseAaffine,DXX:basesv3} for affine quivers) the authors proved that $\mathcal G(Q)$ is a $\Z$-basis in $\mathcal A(Q,\mathbf x)$. The corollary thus follows from Theorem \ref{theorem:generalisation}. 
		\end{proof}

\section{An explicit example}\label{section:example}
	Consider the quiver $Q: 1 \fl 2 \fl 3$ and let $\CC$ be the cluster category of $Q$. The Auslander-Reiten quiver of $\CC$ can be depicted as follows~:
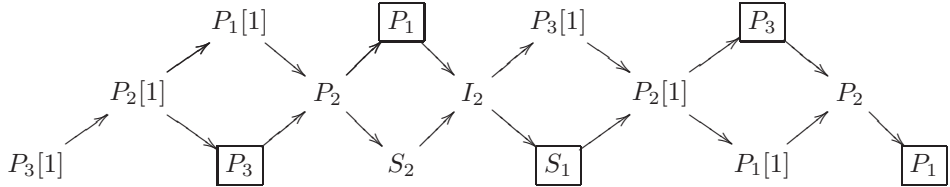
	$$\xymatrix@-3ex{
		&& P_1[1] \ar[rd] && *+[F]{P_1} \ar[rd] && P_3[1] \ar[rd] && *+[F]{P_3} \ar[rd] \\
		& P_2[1] \ar[ru] \ar[ru] \ar[rd]&& P_2 \ar[ru] \ar[ru] \ar[rd]&& I_2 \ar[ru]\ar[rd] && P_2[1] \ar[ru] \ar[rd] && P_2 \ar[rd] \\
		P_3[1] \ar[ru] && *+[F]{P_3}  \ar[ru] && S_2 \ar[ru] && *+[F]{S_1} \ar[ru] && P_1[1] \ar[ru] && *+[F]{P_1}
	}$$
	We choose a cluster-tilting object $T=\mu_{P_2}(\kQ)=P_3 \oplus S_1 \oplus P_1$ and we denote by $T_1=P_3$, $T_2=S_1$ and $T_3=P_1$ its indecomposable summands. The quiver $Q_T$ of the cluster-tilted algebra $\End_{\mathcal C}(T)^{\op}$ is thus the following~:
	$$\xymatrix@-3ex{
		&& T_2 \ar[ld] \\
		Q_T : &T_3 \ar[rr] && T_1. \ar[lu]
	}$$

	Let $\gamma\in K_0(\add T)$. Identifying $K_0(\add T)$ with $\Z^3$, we write $\gamma=(a,b,c)$ for $a,b,c \in \Z$. It follows from Theorem \ref{theorem:stability} that in order to compute $X(\gamma)$ we have to compute the image under $X$ of the cone of a generic morphisms in $\Hom_{\CC}(T_1^{a_1} \oplus T_2^{b_1} \oplus T_3^{c_1}, T_1^{a_0} \oplus T_2^{b_0} \oplus T_3^{c_0})$ where $\gamma_1=(a_1,b_1,c_1)$ and $\gamma_0=(a_0,b_0,c_0)$ have disjoint support and $\gamma = \gamma_0 - \gamma_1$. Thus, we have to the following cases~:

	\subsection*{The case where $\gamma=(a,b,c)$, with $a,b,c \geq 0$}
		In this case, $X(\gamma)$ is given by the image of a generic morphism in $\Hom_{\CC}(0,T_1^a \oplus T_2^b \oplus T_3^c)=0$. Thus, the cone of the (generic) zero morphism is $T_1^a \oplus T_2^b \oplus T_3^c$. It follows that $X(a,b,c)$ is a cluster monomial in the cluster $\ens{X^T_{T_1}, X^T_{T_2}, X^T_{T_3}}$ of $\mathcal A(Q_T,\mathbf x)$.

	\subsection*{The case where $\gamma=(-a,-b,-c)$, with $a,b,c \geq 0$}
		In this case, $X(\gamma)$ is given by the image of a generic morphism in $\Hom_{\CC}(T_1^a \oplus T_2^b \oplus T_3^c,0)=0$. Thus, the cone of the (generic) zero morphism is $T_1^a[1] \oplus T_2^b[1] \oplus T_3^c[1]$. It follows that $X(a,b,c)$ is a cluster monomial in the initial cluster $\ens{x_1, x_2, x_3}$ of $\mathcal A(Q_T,\mathbf x)$.

	\subsection*{The case where $\gamma=(-a,-b,c)$, with $a,b,c \geq 0$}
		In this case, $X(\gamma)$ is given by the image of a generic morphism in $\Hom_{\CC}(T_1^a \oplus T_2^b, T_3^c)$. Let $f$ be a generic morphism in $\Hom_{\CC}(T_1^a \oplus T_2^b, T_3^c)$. Note that there are no morphisms from $T_2 \simeq S_1$ to $T_3 \simeq P_1$ in $\mathcal C$ so that we only have to compute the cone of a generic morphism in $\Hom_{\CC}(T_1^a, T_3^c) \simeq \Hom_{\kQ}(P_3^a, P_1^c)$. A generic morphism in $\Hom_{\kQ}(P_3^a, P_1^c)$ has a zero kernel if $a \leq c$ and has a kernel isomorphic to $P_3^{a-c}$ if $c \leq a$. Similarly, the cokernel is isomorphic to $I_2^c$ if $c \leq a$ and $I_2^c \oplus P_1^{c-a}$ if $a \leq c$. We know that the cone of a morphism $g$ of $\kQ$-modules in $\mathcal C$ is given by $\ker (g) [1] \oplus \coker(g)$. Thus, if $g$ is a generic morphism in $\Hom_{\kQ}(P_3^a, P_1^c)$, we get the triangles
		$$P_3^a \xrightarrow{g} P_1^c \fl I_2^a \oplus P_1^{c-a} \fl P_3^a[1] \textrm{ if } a \leq c,$$
		and 
		$$P_3^a \xrightarrow{g} P_1^c \fl I_2^c \oplus P_3^{a-c}[1] \fl P_3^a[1] \textrm{ if } c \leq a$$
		Since $S_1[1] \simeq \tau S_1 \simeq S_2$ in $\CC$, we obtain the following triangles in the cluster category $\CC$~:
		$$P_3^a \oplus S_1^b \xrightarrow{[g,0]} P_1^c \fl I_2^a \oplus P_1^{c-a} \oplus S_2^b \fl P_3^a[1] \oplus S_2^b \textrm{ if } a \leq c,$$
		and 
		$$P_3^a \oplus S_1^b \xrightarrow{[g,0]} P_1^c \fl I_2^c \oplus P_3^{a-c}[1] \oplus S_2^b \fl P_3^a[1] \oplus S_2^b  \textrm{ if } c \leq a$$
		and $[g,0]$ is generic in $\Hom_{\CC}(T_1^a \oplus T_2^b, T_3^c)$. 
	
		Thus, $X(-a,-b,c)$ is a cluster monomial in the cluster $\ens{X^T_{P_1},X^T_{I_2},X^T_{S_2}}$ if $a \leq c$ and in the adjacent cluster $\ens{X^T_{P_3[1]},X^T_{I_2},X^T_{S_2}}$ if $c \leq a$.

	\subsection*{The case where $\gamma=(-a,b,-c)$, with $a,b,c \geq 0$}
		In this case, $X(\gamma)$ is given by the image of a generic morphism in $\Hom_{\CC}(T_1^a \oplus T_3^c, T_2^b)$. As before, we compute that the cone of such a generic morphism $f$ is given by~:
		$$\cone(f) = \left\{\begin{array}{ll}
			P_3^a[1] \oplus P_2^c[1] \oplus S_1^{b-c} & \textrm{ if } c \leq b, \\
			P_3^a[1] \oplus P_2^b[1] \oplus P_1^{c-b}[1] & \textrm{ if } b \leq c.
		\end{array}\right.$$
		So that $X(-a,b,-c)$ is a cluster monomial in the cluster $\ens{X^T_{P_3[1]},X^T_{P_2[1]},X^T_{S_1}}$ if $c \leq b$ and in the adjacent cluster $\ens{X^T_{P_3[1]},X^T_{P_2[1]},X^T_{P_1[1]}}$ if $b \leq c$.

	\subsection*{The case where $\gamma=(a,-b,-c)$, with $a,b,c \geq 0$}
		In this case, $X(\gamma)$ is given by the image of a generic morphism in $\Hom_{\CC}(T_2^b \oplus T_3^c, T_1^a)$. The cone of such a generic morphism $f$ is given by~:
		$$\cone(f) = \left\{\begin{array}{ll}
			P_1^c[1] \oplus P_2^b \oplus P_3^{a-b} & \textrm{ if } b \leq a, \\
			P_1^c[1] \oplus P_2^a \oplus S_2^{b-a} & \textrm{ if } a \leq b.
		\end{array}\right.$$
		So that $X(a,-b,-c)$ is a cluster monomial in the cluster $\ens{X^T_{P_1[1]},X^T_{P_2},X^T_{P_3}}$ if $b \leq a$ and in the adjacent cluster $\ens{X^T_{P_1[1]},X^T_{P_2},X^T_{S_2}}$ if $a \leq b$.

	\subsection*{The case where $\gamma=(-a,b,c)$, with $a,b,c \geq 0$}
		In this case, $X(\gamma)$ is given by the image of a generic morphism in $\Hom_{\CC}(T_1^a, T_2^b \oplus T_3^c)$. The cone of such a generic morphism $f$ is given by~:
		$$\cone(f) = \left\{\begin{array}{ll}
			S_1^b \oplus I_2^a \oplus P_1^{c-a} & \textrm{ if } c \leq a, \\
			S_1^b \oplus I_2^c \oplus P_3^{a-c}[1] & \textrm{ if } a \leq c.
		\end{array}\right.$$
		So that $X(-a,b,c)$ is a cluster monomial in the cluster $\ens{X^T_{S_1},X^T_{I_2},X^T_{P_1}}$ if $c \leq a$ and in the adjacent cluster $\ens{X^T_{S_1},X^T_{I_2},X^T_{P_3[1]}}$ if $a \leq c$.

	\subsection*{The case where $\gamma=(a,-b,c)$, with $a,b,c \geq 0$}
		In this case, $X(\gamma)$ is given by the image of a generic morphism in $\Hom_{\CC}(T_2^b, T_1^a \oplus T_3^c)$. The cone of such a generic morphism $f$ is given by~:
		$$\cone(f) = \left\{\begin{array}{ll}
			P_1^c \oplus P_2^b \oplus P_3^{a-b} & \textrm{ if } b \leq a, \\
			P_1^c \oplus P_2^a \oplus S_2^{b-a} & \textrm{ if } a \leq b.
		\end{array}\right.$$
		So that $X(a,-b,c)$ is a cluster monomial in the cluster $\ens{X^T_{P_1},X^T_{P_2},X^T_{P_3}}$ if $b \leq a$ and in the adjacent cluster $\ens{X^T_{P_1},X^T_{P_2},X^T_{S_2}}$ if $a \leq b$.

	\subsection*{The case where $\gamma=(a,b,-c)$, with $a,b,c \geq 0$}
		In this case, $X(\gamma)$ is given by the image of a generic morphism in$\Hom_{\CC}(T_3^c, T_1^a \oplus T_2^b)$. The cone of such a generic morphism $f$ is given by~:
		$$\cone(f) = \left\{\begin{array}{ll}
			P_3^a \oplus P_2^c[1] \oplus S_1^{b-c} & \textrm{ if } c \leq b, \\
			P_3^a \oplus P_2^b[1] \oplus P_1^{c-b}[1] & \textrm{ if } b \leq c.
		\end{array}\right.$$
		So that $X(a,b,-c)$ is a cluster monomial in the cluster $\ens{X^T_{P_3},X^T_{P_2[1]},X^T_{S_1}}$ if $c \leq b$ and in the adjacent cluster $\ens{X^T_{P_3},X^T_{P_2[1]},X^T_{P_1[1]}}$ if $b \leq c$.

	It thus follows that $\mathcal G^T(\CC) = \mathcal M(Q_T,\mathbf x)$, illustrating Theorem \ref{theorem:finitetype}.

\section{Conjectures}\label{section:conjectures}
	\subsection{Generic bases}
		The motivation for introducing generic cluster characters is to construct $\Z$-bases in cluster algebras. Corollaries \ref{corol:basetypefini} and \ref{corol:acyclic} provide evidences for the following conjecture~:
		\begin{conj}
			Let $\CC$ be a Hom-finite triangulated 2-Calabi-Yau category such that cluster-tilting subcategories form a cluster structure. Let $T$ be a cluster-tilting object in $\CC$ and assume that $\CC$ has constructible cones with respect to $\add T$-morphisms. Then, $\mathcal G^T(\CC)$ is a $\Z$-linear basis in $\mathcal A(Q_T,\mathbf x)$.
		\end{conj}

	\subsection{Invariance under mutation}
		We conjecture that the set of generic characters is invariant under mutations. More precisely, let $\CC$ be a Hom-finite triangulated 2-Calabi-Yau category such that cluster-tilting subcategories form a cluster structure. Let $T$ be a cluster-tilting object in $\CC$ and let $T'$ be a cluster-tilting object in $\CC$ which is reachable from $T$. Let $\mathbf x=\ens{x_1, \ldots, x_n}$ be the initial cluster in the cluster algebra $\mathcal A(Q_T,\mathbf x)$ and $\mathbf x'=\ens{x_1', \ldots, x_n'}$ be the initial cluster in the cluster algebra $\mathcal A(Q_{T'},\mathbf x')$. We denote by 
		$$\phi : \Q(x_1, \ldots, x_n) \xrightarrow{\sim} \Q(x_1', \ldots, x_n')$$
		the isomorphism sending $x_i$ to $x_i'$ for any $i = 1, \ldots, n$. It induces an isomorphism of $\Z$-algebras between the cluster algebras $\mathcal A(Q_T,\mathbf x)$ and $\mathcal A(Q_{T'},\mathbf x')$ and this isomorphism preserves cluster monomials so that we have a commutative diagram~:
		$$\xymatrix{
			\mathcal M(Q_T,\mathbf x) \ar[r]^{1:1}_{\phi} \ar@{^{(}->}[d] & \mathcal M(Q_{T'},\mathbf x') \ar@{^{(}->}[d] \\
			\Q(x_1, \ldots, x_n) \ar[r]^{\sim}_{\phi} & \Q(x_1', \ldots, x_n') 
		}$$
		We also know from Theorem \ref{theorem:clustermonomials} that the set of cluster monomials is a subset of the set of generic characters. 

		We actually conjecture~:
		\begin{conj}
			With the above notations, if $\CC$ has constructible cones with respect to $\add T$-morphisms and to $\add T'$-morphisms, then the following diagram commutes~:
			$$\xymatrix{
				\mathcal M(Q_T,\mathbf x) \ar[r]^{1:1}_{\phi} \ar@{^{(}->}[d] & \mathcal M(Q_{T'},\mathbf x') \ar@{^{(}->}[d] \\
				\mathcal G^{T}(\CC) \ar@{-->}[r]^{1:1}_{\phi} \ar@{^{(}->}[d] & \mathcal G^{T'}(\CC) \ar@{^{(}->}[d] \\
				\Q(x_1, \ldots, x_n) \ar[r]^{\sim}_{\phi} & \Q(x_1', \ldots, x_n') 
			}$$
		\end{conj}
		For instance, Theorem \ref{theorem:finitetype} proves that this conjecture holds in finite type.

\section*{Acknowledgments}
	This paper was written while the author was at the university of Sherbrooke as a CRM-ISM postdoctoral fellow under the supervision of Ibrahim Assem, Thomas Br\"ustle and Virginie Charette. The author would like to thank Bernhard Keller for interesting discussions on this topic. He would also like to thank an anonymous referee for his interesting suggestions and comments.


\begin{thebibliography}{BMR{\etalchar{+}}06}

\bibitem[ABCP10]{ABCP}
Ibrahim Assem, Thomas Br{\"u}stle, Gabrielle {Charbonneau-Jodoin}, and
  {Pierre-Guy} Plamondon.
\newblock Gentle algebras arising from surface triangulations.
\newblock {\em Algebra and Number Theory}, 4(2):201--229, 2010.

\bibitem[Ami09]{Amiot:clustercat}
Claire Amiot.
\newblock Cluster categories for algebras of global dimension 2 and quivers
  with potential.
\newblock {\em Ann. Inst. Fourier (Grenoble)}, 59(6):2525--2590, 2009.

\bibitem[ASS05]{ASS}
Ibrahim Assem, Daniel Simson, and Andrzej Skowro{\'n}ski.
\newblock {\em Elements of representation theory of Associative Algebras,
  Volume 1: Techniques of Representation Theory}, volume~65 of {\em London
  Mathematical Society Student Texts}.
\newblock Cambridge University Press, 2005.
\newblock MR2197389 (2006j:16020).

\bibitem[BIRS09]{BIRS}
Aslak Buan, Osamu Iyama, Idun Reiten, and Jeanne Scott.
\newblock Cluster structures for 2-{C}alabi-{Y}au categories and unipotent
  groups.
\newblock {\em Compos. Math.}, 145(4):1035--1079, 2009.

\bibitem[BMR{\etalchar{+}}06]{BMRRT}
Aslak Buan, Robert Marsh, Markus Reineke, Idun Reiten, and Gordana Todorov.
\newblock Tilting theory and cluster combinatorics.
\newblock {\em Adv. Math.}, 204(2):572--618, 2006.
\newblock MR2249625 (2007f:16033).

\bibitem[BMR07]{BMR1}
Aslak Buan, Robert Marsh, and Idun Reiten.
\newblock Cluster-tilted algebras.
\newblock {\em Trans. Amer. Math. Soc.}, 359(1):323--332, 2007.
\newblock MR2247893 (2007f:16035).

\bibitem[BMR08]{BMR2}
Aslak Buan, Robert Marsh, and Idun Reiten.
\newblock Cluster mutation via quiver representations.
\newblock {\em Commentarii Mathematici Helvetici}, 83(1):143--177, 2008.

\bibitem[BZ10]{BZ:clustercatsurfaces}
Thomas Br{\"u}stle and Jie Zhang.
\newblock On the cluster category of a marked surface.
\newblock {\em arXiv:1005.2422v2 [math.RT]}, 2010.

\bibitem[CC06]{CC}
Philippe Caldero and Fr{\'e}d{\'e}ric Chapoton.
\newblock Cluster algebras as {H}all algebras of quiver representations.
\newblock {\em Commentarii Mathematici Helvetici}, 81:596--616, 2006.
\newblock MR2250855 (2008b:16015).

\bibitem[CCS06]{CCS1}
Philippe Caldero, Fr{\'e}d{\'e}ric Chapoton, and Ralf Schiffler.
\newblock Quivers with relations arising from clusters ({$A_n$} case).
\newblock {\em Transactions of the AMS}, 358:1347--1354, 2006.
\newblock MR2187656 (2007a:16025).

\bibitem[{Cer}09]{Cerulli:A21}
Giovanni {Cerulli Irelli}.
\newblock Cluster algebras of type {$A_2^{(1)}$}.
\newblock {\em arXiv:0904.2543v3 [math.RT]}, 2009.

\bibitem[CK06]{CK2}
Philippe Caldero and Bernhard Keller.
\newblock From triangulated categories to cluster algebras {II}.
\newblock {\em Annales Scientifiques de l'Ecole Normale Sup{\'e}rieure},
  39(4):83--100, 2006.
\newblock MR2316979 (2008m:16031).

\bibitem[CK08]{CK1}
Philippe Caldero and Bernhard Keller.
\newblock From triangulated categories to cluster algebras.
\newblock {\em Inventiones Mathematicae}, 172:169--211, 2008.
\newblock MR2385670.

\bibitem[CZ06]{CZ}
Philippe Caldero and Andrei Zelevinsky.
\newblock Laurent expansions in cluster algebras via quiver representations.
\newblock {\em Moscow Mathematical Journal}, 6:411--429, 2006.
\newblock MR2274858 (2008j:16045).

\bibitem[DF09]{DF:generalpresentations}
Harm Derksen and Jiarui Fei.
\newblock General presentations of algebras.
\newblock {\em arXiv:0911.4913v1 [math.RA]}, 2009.

\bibitem[DK08]{DK:2CY}
Raika Dehy and Bernhard Keller.
\newblock On the combinatorics of rigid objects in 2-{C}alabi-{Y}au categories.
\newblock {\em Int. Math. Res. Not. IMRN}, (11):Art. ID rnn029, 17, 2008.

\bibitem[Dup08]{Dupont:BaseAaffine}
Gr{\'e}goire Dupont.
\newblock Generic variables in acyclic cluster algebras and bases in affine
  cluster algebras.
\newblock {\em arXiv:0811.2909v2 [math.RT]}, 2008.

\bibitem[Dup11]{Dupont:genericvariables}
Gr{\'e}goire Dupont.
\newblock Generic variables in acyclic cluster algebras.
\newblock {\em Journal of Pure and Applied Algebra}, 215(4):628--641, april
  2011.

\bibitem[DWZ08]{DWZ:potentials}
Harm Derksen, Jerzy Weyman, and Andrei Zelevinsky.
\newblock Quivers with potentials and their representations. {I}. {M}utations.
\newblock {\em Selecta Math. (N.S.)}, 14(1):59--119, 2008.

\bibitem[DXX09]{DXX:basesv3}
Ming Ding, Jie Xiao, and Fan Xu.
\newblock Integral bases of cluster algebras and representations of tame
  quivers.
\newblock {\em arXiv:0901.1937v1 [math.RT]}, 2009.

\bibitem[FK10]{FK}
Changjian Fu and Bernhard Keller.
\newblock On cluster algebras with coefficients and 2-{C}alabi-{Y}au
  categories.
\newblock {\em Trans. Amer. Math. Soc.}, 362:859--895, 2010.

\bibitem[FZ02]{cluster1}
Sergey Fomin and Andrei Zelevinsky.
\newblock Cluster algebras {I}: Foundations.
\newblock {\em J. Amer. Math. Soc.}, 15:497--529, 2002.
\newblock MR1887642 (2003f:16050).

\bibitem[FZ07]{cluster4}
Sergey Fomin and Andrei Zelevinsky.
\newblock Cluster algebras {IV}: Coefficients.
\newblock {\em Compositio Mathematica}, 143(1):112--164, 2007.
\newblock MR2295199 (2008d:16049).

\bibitem[GLS10a]{GLS:generic}
Christof Geiss, Bernard Leclerc, and J.~Schr{\"o}er.
\newblock Generic bases for cluster algebras and the {C}hamber {A}nsatz.
\newblock {\em arXiv:1004.2781v2 [math.RT]}, 2010.

\bibitem[GLS10b]{GLS:KMgroups}
Christof Geiss, Bernard Leclerc, and J.~Schr{\"o}er.
\newblock Kac-moody groups and cluster algebras.
\newblock {\em arXiv:1001.3545v1 [math.RT]}, 2010.

\bibitem[IOTW09]{IOTW}
Kyioshi Igusa, Kent Orr, Gordana Todorov, and Jerzy Weyman.
\newblock Cluster complexes via semi-invariants.
\newblock {\em Compos. Math.}, 145:1001--1034, 2009.

\bibitem[Kac80]{Kac:infroot1}
Victor Kac.
\newblock Infinite root systems, representations of graphs and invariant
  theory.
\newblock {\em Inventiones Mathematicae}, 56:57--92, 1980.
\newblock MR0557581 (82j:16050).

\bibitem[Kac82]{Kac:infroot2}
Victor Kac.
\newblock Infinite root systems, representations of graphs and invariant theory
  {II}.
\newblock {\em Journal of algebra}, 78:163--180, 1982.
\newblock MR0677715 (85b:17003).

\bibitem[Kel05]{K}
Bernhard Keller.
\newblock On triangulated orbit categories.
\newblock {\em Documenta Mathematica}, 10:551--581, 2005.
\newblock MR2184464 (2007c:18006).

\bibitem[KR07]{KR:clustertilted}
Bernhard Keller and Idun Reiten.
\newblock Cluster-tilted algebras are {G}orenstein and stably {C}alabi-{Y}au.
\newblock {\em Adv. Math.}, 211(1):123--151, 2007.

\bibitem[KR08]{KR:acyclic}
Bernhard Keller and Idun Reiten.
\newblock Acyclic {C}alabi-{Y}au categories.
\newblock {\em Compositio Mathematicae}, 144(5):1332--1348., 2008.
\newblock MR2457529.

\bibitem[Lam]{Lampe:Kronecker}
Philipp Lampe.
\newblock A quantum cluster algebra of {K}ronecker type and the dual canonical
  basis.
\newblock {\em Int. Math. Res. Not. IMRN}.
\newblock to appear.

\bibitem[Nak10]{Nakajima:cluster}
Hiraku Nakajima.
\newblock Quiver varieties and cluster algebras.
\newblock {\em arXiv:0905.0002v5 [math.QA]}, 2010.

\bibitem[Pal08]{Palu}
Yann Palu.
\newblock Cluster characters for 2-{C}alabi-{Y}au triangulated categories.
\newblock {\em Ann. Inst. Fourier (Grenoble)}, 58(6):2221--2248, 2008.

\bibitem[Pal09]{Palu:multiplication}
Yann Palu.
\newblock Cluster characters {II}: {A} multiplication formula.
\newblock {\em arXiv:0903.3281v2 [math.RT]}, 2009.

\bibitem[Pla10]{Plamondon:ClusterAlgebras}
{Pierre-Guy} Plamondon.
\newblock Cluster algebras via cluster categories with infinite-dimensional
  morphism spaces.
\newblock {\em arXiv:1004.0830v2 [math.RT]}, 2010.

\bibitem[Sch92]{Schofield:generalrepresentations}
Aidan Schofield.
\newblock General representations of quivers.
\newblock {\em Proc. London Math. Soc.}, 65(3):46--64, 1992.
\newblock MR1162487 (93d:16014).

\bibitem[SZ04]{shermanz}
Paul Sherman and Andrei Zelevinsky.
\newblock Positivity and canonical bases in rank 2 cluster algebras of finite
  and affine types.
\newblock {\em Mosc. Math. J.}, 4:947--974, 2004.
\newblock MR2124174 (2006c:16052).

\end{thebibliography}

\newcommand{\etalchar}[1]{$^{#1}$}

\end{document}